\documentclass[
final
]{dmtcs-episciences}
\usepackage[utf8]{inputenc}
\usepackage{clm-macros}
\usepackage{amsmath}

\usepackage[round]{natbib}
\usepackage{MnSymbol}
\usepackage{subcaption}
\usepackage{datetime}

\usepackage{tikz}
\usetikzlibrary{calc,scopes}
\usetikzlibrary{backgrounds}
\usetikzlibrary{decorations}
\usetikzlibrary{decorations.pathmorphing}
\usetikzlibrary{arrows}
\tikzstyle{every node}=[circle, draw, fill=white,inner sep=0pt, minimum width=4pt]
\tikzstyle{nodelabel}=[rounded corners,fill=none,inner sep=5pt,draw=none]
\tikzstyle{matching} = [ultra thick]
\tikzset{snake/.style={decorate, decoration=snake}}

\def\Notname{Notation}
\def\casname{Case}
\def\cnjname{Conjecture}
\def\corname{Corollary}
\def\Defname{Definition}
\def\exename{Exercise}
\def\exaname{Example}
\def\lemname{Lemma}
\def\listassertionname{List of Assertions}
\def\prbname{Decision Problem}
\def\prpname{Proposition}
\def\proofOfname{\proofname{} of}
\def\ssbname{Case}
\def\subname{Case}
\def\thmname{Theorem}

\newcommand{\mcg}{matching covered graph}
\newcommand{\mc}{matching covered}

\newcommand{\ce}{cycle-extendable}
\newcommand{\pema}{perfect matching}

\newcommand{\NT}{Norine-Thomas}
\newcommand{\bw}{half biwheel}
\newcommand{\tvc}{\mbox{$2$-vertex-cut}}

\title[Planar cycle-extendable graphs]{Planar cycle-extendable graphs\\\small{Three of the authors dedicate this work to their coauthor, the late Ajit A Diwan}}
\author[Dalwadi, Pause, Diwan and Kothari]{Aditya Y Dalwadi\affiliationmark{1}\thanks{Supported by  IITM Pravartak Technologies Foundation.}
  \and Kapil R Shenvi Pause\affiliationmark{2}\\
  \and Ajit A Diwan \affiliationmark{3}
  \and Nishad Kothari \affiliationmark{1} \thanks{Supported by IC\&SR IIT Madras.}}
  \affiliation{
  IIT Madras, Chennai, India\\
  Chennai Mathematical Institute, Chennai, India\\
  IIT Bombay, Mumbai, India}

\begin{document}

\publicationdata
{vol. 27:2}
{2025}
{13}
{10.46298/dmtcs.13929}
{2024-07-14; 2024-07-14; 2025-03-26}
{2025-04-30}

  \maketitle

\begin{abstract}
For most problems pertaining to perfect matchings, one may restrict attention to {\em \mcg s} --- that is, connected nontrivial graphs with the property that each edge belongs to some perfect matching. There is extensive literature on these graphs that are also known as {\em $1$-extendable graphs} (since each edge extends to a perfect matching) including an ear decomposition theorem due to Lov\'asz and Plummer.

 A cycle $C$ of a graph $G$ is {\em conformal} if $G-V(C)$ has a perfect matching; such cycles play an important role in the study of perfect matchings, especially when investigating the Pfaffian orientation problem. A \mcg~$G$ is {\em\ce} if --- for each even cycle $C$ --- the cycle $C$ is conformal, or equivalently, each perfect matching of $C$ extends to a perfect matching of $G$, or equivalently, $C$ is the symmetric difference of two perfect matchings of $G$, or equivalently, $C$ extends to an ear decomposition of~$G$. In the literature, these are also known as {\em cycle-nice} or as {\em $1$-cycle resonant} graphs.


Zhang, Wang, Yuan, Ng and Cheng, 2022, provided a characterization of claw-free \ce~graphs. Guo and Zhang, 2004, and independently Zhang and Li, 2012, provided characterizations of bipartite planar \ce~graphs. In this paper, we establish a characterization of all planar \ce~graphs --- in terms of $K_2$ and four infinite families. 

\end{abstract}

\section{Cycle-extendability}
All graphs considered in this paper are loopless. However, we allow multiple/parallel edges. For graph-theoretic notation and terminology, we follow~\cite{bomu08}, whereas for terminology pertaining to matching theory, we follow~\cite{lumu22}.

A graph is said to be {\em matchable} if it has a perfect matching. Most problems pertaining to the study of perfect matchings may be reduced to {\em \mcg s} --- that is, nontrivial connected graphs with the property that each edge lies in some perfect matching. There is extensive literature on \mcg s.  In~\cite{lopl86}, they are referred to as {\em $1$-extendable graphs}, since each edge extends to a perfect matching. There is also literature on {\em $k$-extendable graphs} --- \mcg s with the additional property that each matching of cardinality $k$ extends to a perfect matching; see~\cite{plummer91}. In a similar spirit, a \mcg\ $G$ is {\em \mbox{\ce}} if, for each even cycle~$C$, either perfect matching of $C$ extends to a perfect matching of $G$. This leads us to the following decision problem that is not known to be in \textbf{NP}.

 \begin{prb}
 \label{main-problem}
     Given a \mcg\ $G$, decide whether $G$ is \ce.
 \end{prb}

  Before stating our contributions and related prior work on the above problem, let us take a closer look at conformality and cycle-extendability. A cycle $C$ of a matchable graph $G$ is a {\em conformal cycle} if the graph $G-V(C)$ is matchable. Thus, a \mcg~is \ce\ if and only if each even cycle is conformal. From this viewpoint, it is easy to see that Decision Problem \ref{main-problem} belongs to {\bf co-NP}. 
  
  In Figure \ref{fig:cube}, the $6$-cycle denoted by the dashed line is not conformal, and using the symmetries of the Cube graph, the reader may verify that there are precisely four such \mbox {$6$-cycles }, whereas every other cycle is conformal. In particular, the Cube graph is not \ce. Figures \ref{fig:bipartite-ce} and \ref{fig:nonbipartite-ce} depict examples of bipartite and nonbipartite \ce\ graphs, respectively.

\begin{figure}[!htb]
    \centering
    \begin{subfigure}{0.3\textwidth}
            \centering
             \begin{tikzpicture}[every node/.style={draw=black, circle, scale=0.66}, scale=0.7]
                \node (a1) at (3,3){} ;  
                \node[fill=black] (a2) at (3,5){};
                \node[fill=black,label=below:{$v$}] (a4) at (5,3){};
                \node (a3) at (5,5){};
                \node[fill=black] (a5) at (1.5,1.5){} ;  
                \node (a6) at (1.5,6.5){};
                \node (a8) at (6.5,1.5){};4
                \node[fill=black] (a7) at (6.5,6.5){};
                \draw[dashed] (a5) -- (a1);\draw (a6) -- (a2);\draw[dashed] (a7) -- (a3); \draw (a8) -- (a4);
                \draw[dashed] (a1) -- (a2);
                \draw[dashed] (a2) -- (a3);
                \draw (a3) -- (a4);
                \draw (a4) -- (a1);
                \draw[] (a5) -- (a6); 
                \draw (a6) -- (a7);
                \draw[dashed] (a7) -- (a8);
                \draw[dashed] (a8) -- (a5);
         \end{tikzpicture}
         \caption{}
        \label{fig:cube}
    \end{subfigure}%
    \hfill
    \begin{subfigure}{0.3\textwidth}
        \centering
          \begin{tikzpicture}[every node/.style={draw=black, circle,scale=0.66}, scale=0.5]
               \node[fill=black,text=white] (a1) at (0,5){};
               \node (a2) at (1.5,5){};
               \node[fill=black,text=white] (a3) at (3,5){};
               \node (a4) at (4.5,5){};
               \node[fill=black,text=white] (a5) at (6,5){};
                \node (a8) at (3,0){};
                \draw (a1) -- (a8) -- (a3);
                \draw (a1) -- (a2) -- (a3) -- (a4) -- (a5);
                \draw (a5) -- (a8);
            \end{tikzpicture}
        \caption{}
        \label{fig:bipartite-ce}
     \end{subfigure}
     \hfill
    \begin{subfigure}{0.3\textwidth}
            \centering
            \begin{tikzpicture}[every node/.style={draw=black, circle,scale=0.66}, scale=0.5]
                \node (a6) at (0,1.5) {};
                \node (a7) at (-0.5,5.75) {};
                \node (a8) at (3,8) {};
                \node (a9) at (6.5,5.75) {};
                \node (a10) at (5.75,1.5) {};
                \node (a1) at (3,4.5){};
                
                \draw (a6) -- (a7);
                \draw (a1) -- (a6);
                \draw (a1) -- (a7);
                \draw (a1) -- (a9);
                \draw (a1) -- (a10);
                \draw (a7) -- (a8);
                \draw (a8) -- (a9);              
                \draw[] (a9) -- (a10);
                \draw (a10) -- (a6);
        \end{tikzpicture}
        \caption{}
        \label{fig:nonbipartite-ce}
     \end{subfigure}
     \caption{(a) the Cube graph is not \ce; (b) a bipartite \ce\ graph; (c) a nonbipartite \ce\ graph $W_5^{-}$}
 \end{figure}
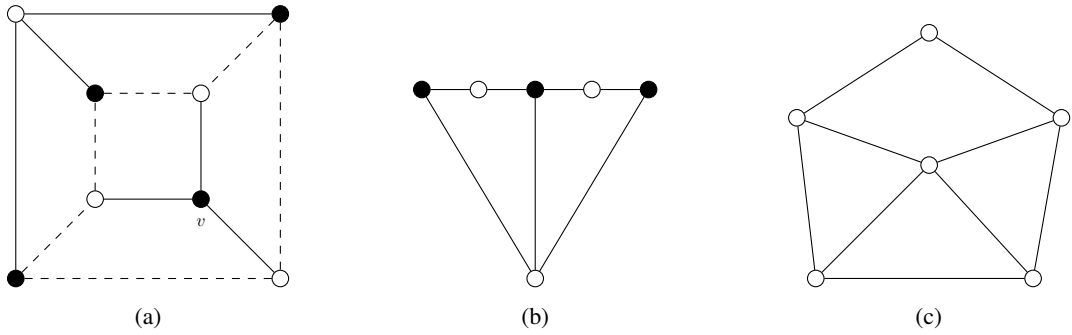
 
 Observe that, in a matchable graph, a cycle is conformal if and only if it may be expressed as the symmetric difference of two perfect matchings. It is for this reason that they are also known as {\em alternating cycles}; see~\cite{cali08}. It is easily observed that a connected matchable graph, distinct from $K_2$, is matching covered if and only if each edge belongs to a conformal cycle.~\cite{litt75} proved the stronger statement that, in a \mcg, any two edges belong to a common conformal cycle. There are several other terms used in the literature for conformal cycles : {\em well-fitted} cycles in~\cite{mccu04}; {\em nice} cycles in~\cite{lova87}; and {\em central} cycles in~\cite{rst99}.

Zhang, Wang, Yuan, Ng and Cheng~\cite{zwync22} provided a characterization of claw-free \ce\ (aka cycle-nice) \mcg s; their work implies that Decision Problem \ref{main-problem} belongs to {\bf NP} as well as {\bf P} for claw-free graphs. 

\cite{Guozhang04}, and independently~\cite{Zhli12}, provided characterizations of bipartite planar \ce~(aka 1-cycle resonant) graphs. The second author discovered a couple of these characterizations independently, and they appear in their MSc thesis~\cite{paus22}; they also made partial progress towards characterizing planar \ce~graphs.

In this paper, we characterize all planar \ce\ graphs and our result implies that Decision Problem \ref{main-problem} belongs to {\bf NP} as well as {\bf P} for planar graphs. The results reported here comprise a proper subset of those that appear in the MS~thesis~\cite{dalwadi25} of the first author.

\bigskip
\noindent{\bf Organization of this paper}
\bigskip

In Section~\ref{Section1.1}, we discuss the ear decomposition theory for \mcg s that provides an alternative definition of \ce~graphs that motivates their investigation. In Section~\ref{sec:irreducible-graphs}, we discuss series and parallel reductions that help us in restricting ourselves to irreducible graphs --- that is, simple graphs whose vertices of degree two comprise a stable set. 

In Section~\ref{section-3}, we describe the tight cut decomposition theory for \mcg s and its applications towards characterizing \ce~graphs. In particular, in Section~\ref{Section-3.3}, we use a result of
~\cite{cali08} to deduce that every planar \ce~irreducible graph, except $K_2$, either has a vertex of degree two or otherwise is a ``brick'' --- a special class of $3$-connected nonbipartite \mcg s.

In Section~\ref{sec:k23-free}, we first discuss a necessary condition ($K_{2,3}$-freeness) for a planar \mcg~to be \ce, and then a necessary condition for a $3$-connected graph to be $K_{2,3}$-free. In Section~\ref{section-4}, we use results of Section~\ref{sec:k23-free} and the well-known brick generation theorem of~\cite{noth07} to characterize planar \ce~bricks.

In Section~\ref{section-5},  we describe four infinite families of nonbipartite planar \ce\ irreducible graphs using a special class of bipartite \ce\ graphs called ``\bw s''. Finally, in Section~\ref{section-6}, we prove that every planar \ce~irreducible graph is either $K_2$ or is a member of one of these four families; this proves a conjecture of the third author, and places Decision Problem~\ref{main-problem} in {\bf NP} as well as in {\bf P} for planar graphs.

\subsection{Ear decompositions of \mcg s} 
\label{Section1.1}
Observe that \mcg s, except $K_2$, are $2$-connected. The ear decomposition theory of $2$-connected graphs, due to~\cite{whit33}, is well-known. We now discuss a refinement of this theory, due to Hetyei, that is applicable to the subclass of bipartite \mcg s; see \cite{lopl86}.

 Given a graph $G$ and a (proper) subgraph $H$, {\em an ear} (aka a {\em single ear}) of $H$ in $G$ is an odd path $P$ whose ends are in $V(H)$ but is otherwise disjoint with $H$. For a bipartite~graph~$G$, a sequence of subgraphs $(G_0, G_1, \dots, G_r)$ is called a {\em bipartite ear decomposition} of $G$ if (i)~$G_0$ is an (even) cycle, (ii)~$G_r=G$, and (iii)~$G_{i+1}=G_i \cup P_i$ where $P_i$ is an ear  of $G_i$ (in $G$) for each $i\in \{0,1,\ldots,r-1\}$. Figure~\ref{ear-decom-K_33} shows a bipartite ear decomposition of a bipartite \mcg\ where each ear is denoted by thick lines.

 \begin{figure}[!htb]
    \centering
   \begin{tikzpicture}[every node/.style={draw=black, circle,scale=0.66}, scale=0.5]
               \node[fill=black,text=white] (a1) at (0,5){};
               \node (a2) at (1.5,5){};
               \node[fill=black,text=white] (a3) at (3,5){};
               \node (a4) at (4.5,5){};
               \node[fill=black,text=white] (a5) at (6,5){};
                \node (a8) at (3,0){};
                \draw (a1) -- (a8);
                \draw (a1) -- (a2) -- (a3);
                 \draw (a3) -- (a4) -- (a5) --(a8);
    \end{tikzpicture}
 \hspace*{20pt}
\begin{tikzpicture}[every node/.style={draw=black, circle,scale=0.66}, scale=0.5]
               \node[fill=black,text=white] (a1) at (0,5){};
               \node (a2) at (1.5,5){};
               \node[fill=black,text=white] (a3) at (3,5){};
               \node (a4) at (4.5,5){};
               \node[fill=black,text=white] (a5) at (6,5){};
                \node (a8) at (3,0){};
                \draw (a1) -- (a8);
                \draw[ultra thick] (a8) -- (a3);
                \draw (a1) -- (a2) -- (a3);
                 \draw (a3) -- (a4) -- (a5) --(a8);
    \end{tikzpicture}
 \hspace*{20pt}
\begin{tikzpicture}[every node/.style={draw=black, circle,scale=0.66}, scale=0.5]
               \node[fill=black,text=white] (a1) at (0,5){};
               \node (a2) at (1.5,5){};
               \node[fill=black,text=white] (a3) at (3,5){};
               \node (a4) at (4.5,5){};
               \node[fill=black,text=white] (a5) at (6,5){};
               \node (a6) at (7.5,5){};
               \node[fill=black,text=white] (a7) at (9,5){};
                \node (a8) at (4.5,0){};
                \draw (a1) -- (a8) -- (a3);
                \draw (a1) -- (a2) -- (a3);
                 \draw (a3) -- (a4) -- (a5) --(a8);
                 \draw[ultra thick] (a5) -- (a6) -- (a7) --(a8);
    \end{tikzpicture}
    \caption{a bipartite ear decomposition of a bipartite \mcg}
    \label{ear-decom-K_33}
\end{figure}
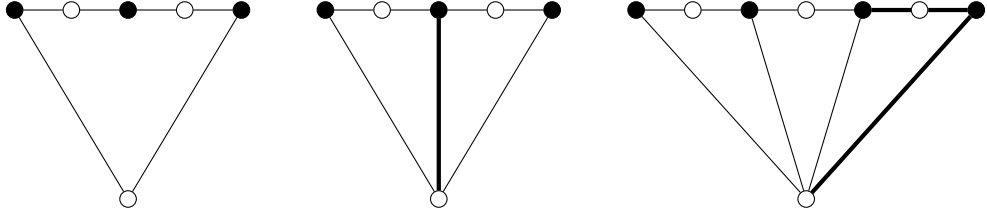

A subgraph $H$ of a graph $G$ is {\em conformal} if the graph $G-V(H)$ is matchable. Given any bipartite ear decomposition of a bipartite graph, it is easily observed that each subgraph in the sequence is conformal. Furthermore, the following theorem due to Hetyei, implies that each subgraph is also \mc\ and establishes the converse as well.

\begin{thm}{\sc [Bipartite Ear Decomposition Theorem]}\newline
\label{thm:odd-ear-decomposition}
A bipartite graph $G$, distinct from $K_2$, is matching covered if and only if each of its conformal cycles extends to a bipartite ear decomposition of $G$.
\end{thm}

For the Cube graph, depicted in Figure \ref{fig:cube}, the reader may observe that each conformal cycle extends to a bipartite ear decomposition; furthermore, the non-conformal $6$-cycle (shown using dashed lines) does not extend to a bipartite ear decomposition.

Theorem \ref{thm:odd-ear-decomposition} implies that any bipartite \mcg\ (distinct from $K_2$)  can be constructed in a straightforward manner, from any conformal cycle, by adding a single ear at a time so that each subgraph is also \mcg.  Unfortunately, in the case of nonbipartite \mcg s, one can not restrict to the addition of single ears. For instance, in the case of $K_4$, we must start from $C_4$; now observe that we must add the remaining two edges simultaneously in order to get a bigger \mc\ subgraph (that is, $K_4$ itself). \cite{lopl86} proved the surprising result that every \mcg\ may be constructed, from any conformal cycle, by adding either a single ear or a ``double ear'' at a time so that each subgraph is also \mc. We state this more formally below.

Given a graph $G$ and (proper) subgraph $H$, a {\em double ear} of $H$ in $G$ is a pair of vertex-disjoint single ears of $H$ (in $G$). For a \mcg~$G$, a sequence of \mc\ subgraphs $(G_0, G_1, \dots, G_r)$ is called an {\em ear~decomposition} of $G$ if (i)~$G_0$ is an even cycle, (ii)~$G_r=G$, and (iii)~$G_{i+1}=G_i \cup R_i$ where $R_i$ is either a single or a double ear  of $G_i$ (in $G$) for each $i\in \{0,1,\ldots,r-1\}$. As in the case of bipartite ear decompositions, it is easy to see that each subgraph in an ear decomposition is also conformal. We are now ready to state the aforementioned theorem of Lov\'asz and Plummer.

\begin{thm}{\sc [Ear Decomposition Theorem]}\newline
\label{thm1.4}
 Each conformal cycle of a \mcg\ $G$ extends to an ear {\mbox decomposition} of $G$.
\end{thm}

  Figure \ref{fig:ear-dec-petersen} shows an ear decomposition of the nonbipartite graph~$R_8^{-}$ that will play an important role in our work, and is obtained from the bicorn $R_8$, shown in Figure~\ref{fig:$K_{2,3}$-based graph}, by deleting an edge; each ear addition is denoted by a thick line.
  \begin{figure}[!htb]
      \centering
       \begin{tikzpicture}[every node/.style={draw=black, circle,scale=0.66}, scale=0.5]
               \node[fill=black] (a1) at (2,2){};
               \node (a2) at (3.5,2){};
               \node[fill=black] (a3) at (5,2){};
               \node[draw=none] (a4) at (7,0){};
               \node (a4) at (0,4){};
               \node[fill=black] (a5) at (3.5,4){};
               \node (a6) at (7,4){};
               \draw (a1) -- (a2) -- (a3) -- (a6) -- (a5) -- (a4) -- (a1); 
    \end{tikzpicture}
    \hspace*{20pt}
    \begin{tikzpicture}[every node/.style={draw=black, circle,scale=0.66}, scale=0.5]
               \node[fill=black] (a1) at (2,2){};
               \node (a2) at (3.5,2){};
               \node[fill=black] (a3) at (5,2){};
               \node (a7) at (7,0){};
               \node[fill=black] (a8) at (0,0){};
               \node (a4) at (0,4){};
               \node[fill=black] (a5) at (3.5,4){};
               \node (a6) at (7,4){};
               \draw (a1) -- (a2) -- (a3) -- (a6) -- (a5) -- (a4) -- (a1); 
               \draw[ultra thick] (a4) -- (a8) -- (a7) -- (a3);
    \end{tikzpicture}
    \hspace*{20pt}
    \begin{tikzpicture}[every node/.style={draw=black, circle,scale=0.66}, scale=0.5]
               \node (a1) at (2,2){};
               \node (a2) at (3.5,2){};
               \node (a3) at (5,2){};
               \node (a7) at (7,0){};
               \node (a8) at (0,0){};
               \node (a4) at (0,4){};
               \node (a5) at (3.5,4){};
               \node (a6) at (7,4){};
               \draw (a1) -- (a2) -- (a3) -- (a6) -- (a5) -- (a4) -- (a1); 
               \draw (a4) -- (a8) -- (a7) -- (a3);
               \draw[ultra thick] (a8) -- (a1);
               \draw[ultra thick] (a6) -- (a7);
    \end{tikzpicture}
      \caption{an ear decomposition of $R_8^{-}$}
      \label{fig:ear-dec-petersen}
  \end{figure}
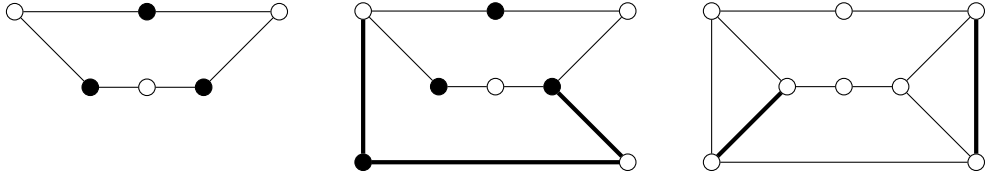

The statement of Theorem \ref{thm1.4} begs the question: for which \mcg s, is it possible to extend each even cycle to an ear decomposition? Clearly, this is possible only for those \mcg s that are \ce. This provides an alternative motivation for characterizing \ce\ graphs.

 \begin{prop}{\sc [An Alternative Viewpoint Of Cycle-Extendability]}\newline
     A \mcg\ $G$ is \ce\ if and only if each of its even cycles extends to an ear decomposition of $G$. \qed
 \end{prop}

We conclude this section by describing a special class of nonbipartite \mcg s that play a crucial role in the theory of \mcg s, as well as in our work. Before doing so, we remark that, in an ear decomposition of a \mcg, a double ear is added only when neither of the constituent single ears can be added in order to obtain a matching covered graph. 

Let $G$ be a nonbipartite \mcg\ and $(G_0,G_1,G_2,\ldots,G_r=G)$ be an ear decomposition. Note that $G_0$ is an even cycle. Let $G_k$ denote the first nonbipartite subgraph in this sequence. It follows from our above remark that $G_k$ is also the first subgraph in the sequence that is obtained by adding a double ear, say $R_{k-1}$, to the previous (bipartite) subgraph $G_{k-1}$. We say that $G_{k}$ is a {\em near-bipartite graph}. Below, we provide an alternative definition that is independent of the ear decomposition.

For a \mcg\ $G$, an odd path $P$ is a {\em removable single ear} of $G$ if each internal vertex of $P$ (if any) has degree two in $G$ and the graph $G-P$ is also \mc. (Here, by $G-P$, we mean the graph obtained from $G$ by deleting all of the edges and internal vertices of $P$.) Likewise, a pair of vertex-disjoint odd paths $R:=\{P_1,P_2\}$ is a {\em removable~double~ear} of $G$ if each internal vertex of $P_1$ as well as $P_2$ has degree two in $G$ and the graph $G-P_1-P_2$ is \mc. A {\em near-bipartite graph} is a \mcg, say $G$, that has a removable double ear $R:=\{P_1,P_2\}$ such that $G-R:=G-P_1-P_2$ is bipartite. In this sense, near-bipartite graphs comprise a subset of nonbipartite \mcg s whose members are closest to being bipartite. 


In the next section, we discuss a reduction of the Decision Problem~\ref{main-problem} to a more restricted class of \mcg s.

\subsection{Irreducible graphs}
\label{sec:irreducible-graphs}
Given a graph $G$ and a parallel edge $e$, we say that $G-e$ is obtained from $G$ by an application of {\em parallel reduction}. Observe that $G$ is \mc~if and only if $G-e$ is \mc; furthermore, $G$ is \ce~if and only if $G-e$ is \ce. 

Let $G$ be a graph that has a path $P:=wxyz$ of length three, each of whose internal vertices, $x$~and~$y$, is of degree two in $G$. Let $J$ denote the graph obtained from $G$ by replacing the path $P$ with a single edge $e$ joining $w$ and $z$. That is, $J:=G-P+e$. We say that $J$ is obtained from $G$ by an application of {\em series reduction}. The reader may observe that $G$ is \mc~if and only if $J$ is \mc~except when $J$ is $K_2$; furthermore, $G$ is \ce~if and only if $J$ is \ce.

A graph $G$ is {\em irreducible} if it is simple and its degree two vertices comprise a stable set. The above observations prove the following. 

\begin{prp}
\label{prop:irreducible-reduction}
    Let $H$ be an irreducible (\mc) graph that is obtained from a \mcg~$G$ by repeated applications of series and parallel reductions. Then $G$ is \ce~if and only if $H$ is \ce. \qed
\end{prp}

Thus, in order to settle the complexity status of Decision Problem~\ref{main-problem}, it suffices to focus on the following decision problem.

\begin{prb}
    Given an irreducible \mcg~$G$, decide whether $G$ is \ce~or not.
\end{prb}

In the next section, we discuss a necessary matching-theoretic condition for a graph to be \ce. In order to do so, we shall require some concepts from the theory of \mcg s.
\section{Applications of the tight cut decomposition theory}
\label{section-3}
For a nonempty proper subset $X$ of the vertices of a graph $G$, we denote by $\partial(X)$ {\em the cut associated with $X$}, that is, the set of all edges that have one end in $X$ and the other end in $\overline{X} := V (G) - X$. We refer to $X$ and $\overline{X}$ as the {\em shores} of $\partial(X)$. For a vertex $v$, we simplify the notation $\partial(\{v\})$ to $\partial(v)$, and we refer to such a cut as {\em trivial}. 

For a cut $\partial(X)$ of a graph $G$, we denote by $G/(X \rightarrow x)$, or simply by $G/X$, the graph obtained from $G$ by shrinking the shore~$X$ to a single vertex $x$ called the {\em contraction vertex}. The graph $G/\overline{X}$ is defined analogously. The graphs $G/\overline{X}$ and $G/X$ are called the {\em $\partial(X)$-contractions of $G$}. Observe that each edge in either $\partial(X)$-contraction corresponds to an edge in $G$. Conversely, each edge of $G$ that is not in $\partial(X)$ corresponds to an edge in precisely one of the two $\partial(X)$-contractions, whereas each edge of $G$ that is in $\partial(X)$ corresponds to an edge in both $\partial(X)$-contractions; it is customary and convenient to use the same label for all these edges.  For the graph shown in Figure~\ref{fig:not a cycle-extendable graph}, the blue line indicates a cut; one of the contractions is $K_4$, whereas the other contraction is obtained from $K_{3,3}$ by replacing any vertex by a triangle.

\begin{figure}[!htb]
    \centering
         \begin{tikzpicture}[every node/.style={draw=black, circle,scale=0.66,fill=white}, scale=0.5,rotate=180]
                \node (a1) at (0,0){};
                \node (a2) at (0,5){};
                \node (a3) at (-1,7){};
                \node (a4) at (1,7){};
                \draw[ultra thick] (a2) -- (a3);
                \draw (a3) -- (a4) -- (a2);
                
                \node (a5) at (8,0){};
                \node (a6) at (8,5){};
                \node (a7) at (7,7){};
                \node (a8) at (9,7){};
                \draw (a6) -- (a7);
                \draw[ultra thick] (a7) -- (a8);
                \draw (a8) -- (a6);

                \node (a9) at (16,0){};
                \node (a10) at (16,6){};
                \draw[ultra thick] (a10) -- (a9);
                \draw (a10) -- (a1);
                \draw[ultra thick] (a10) -- (a5);

                \draw (a6) -- (a5);
                \draw[ultra thick] (a7) -- (a1);
                \draw[ultra thick] (a8) -- (a9);
                
                \draw[ultra thick] (a2) -- (a5);
                \draw (a4) -- (a9);
                \draw[in=150, out=240,ultra thick] (a3) to (a1);
                \node[draw=none] (a21) at (-3,8){};
                \node[draw=none] (a22) at (3,8){};
                \draw[in=270, out=270,ultra thick,looseness=2.2,blue] (a21) to (a22);
                \node[draw=none] (a31) at (5,8){};
                \node[draw=none] (a32) at (11,8){};
                \draw[in=270, out=270,ultra thick,looseness=2.2,blue] (a31) to (a32);
        \end{tikzpicture}    
        
 \caption{a \mcg~and its nontrivial tight cuts}
  \label{fig:not a cycle-extendable graph}
\end{figure}
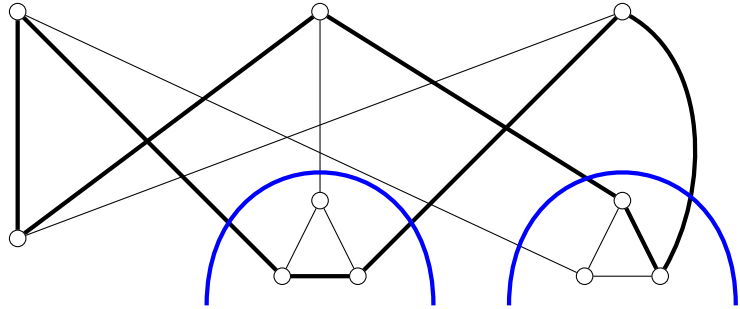

\subsection{Tight cut decomposition, bricks and braces}
\label{Section:tight-cut}
Let $G$ be a \mcg. A cut $\partial(X)$ is a {\em tight cut} if $|M \cap \partial(X)| = 1$ for every perfect matching $M$. In Figure~\ref{fig:not a cycle-extendable graph}, the blue lines indicate nontrivial tight cuts. It is easy to see that if $\partial(X)$ is a nontrivial tight cut then each $\partial(X)$-contraction is a \mcg\ that has fewer vertices than $G$. If either of the $\partial(X)$-contractions has a nontrivial tight cut, then that graph can be further decomposed into smaller \mcg s. We may repeat this procedure until we obtain a list of \mcg s --- each of which is free of nontrivial tight cuts. This process is known as the {\em tight cut decomposition procedure}.

A \mcg\ free of nontrivial tight cuts is called a {\em brace} if it is bipartite; otherwise, it is called a {\em brick}. Thus, an application of the tight cut decomposition procedure to a \mcg\ results in a list of bricks and braces. For the graph shown in Figure~\ref{fig:not a cycle-extendable graph}, an application of the tight cut decomposition procedure yields two copies of $K_4$ and one copy of $K_{3,3}$. We now make a simple observation pertaining to bricks and braces.

Let $S:=\{u,v\}$ denote a \tvc\ in a \mcg\ $G$ on six or more vertices. Let $J$ denote a component of $G-S$. If $J$ is a nontrivial odd component then $\partial(V(J))$ is a nontrivial tight cut; whereas, if $J$ is an even component then $\partial(V(J) \cup \{u\})$ is a nontrivial tight cut. This immediately proves the following well-known fact. 
 
\begin{prop}
\label{prp:brick-brace-3-connected}
 Every simple brick/brace, except $K_2$ and $C_4$, is $3$-connected.\qed   
\end{prop}

A \mcg\ may admit several applications of the tight cut decomposition procedure. However,~\cite{lova87} proved the following remarkable result.
\begin{thm}{\sc [Unique Tight Cut Decomposition Theorem]}\newline
\label{THE UNIQUE DECOMPOSITION THEOREM}
     Any two applications of the tight cut decomposition procedure to a \mcg~yield the same list of bricks and braces (up to multiplicities of edges).
\end{thm}

The following characterization of planar \ce\ braces is an immediate consequence of Proposition~\ref{prp:brick-brace-3-connected} and a result of~\cite[Corollary 2.3]{KlavzarSalem12}.

\begin{cor}
    \label{brace-characterize}
    The only simple planar \ce\ braces are $K_2$ and $C_4$. \qed
\end{cor}

We now switch our attention to proving the following lemma that relates cycle-extendability with tight cuts. 

\begin{lem}
\label{prp:cycle-extendability-inherited-through-tight-cuts}
    Let $G$ be a \mcg\ and let $\partial(X)$ denote a nontrivial tight cut. If $G$ is \ce\ then both $\partial(X)$-contractions of $G$ are also \ce.
\end{lem}
\begin{proof}
We let $G_1:=G/(X \rightarrow x)$ and $G_2:=G/(\overline{X} \rightarrow \overline{x})$.
Assume that $G$ is \ce. It suffices to prove that $G_1$ is \ce. To this end, let $Q_1$ denote an even cycle of~$G_1$.

If $x \notin V(Q_1)$ then $Q_1$ is a cycle of~$G$. We let $M$ denote a \pema\ of~$G-V(Q_1)$.
Observe that, since $M$ extends to a \pema\ of~$G$, and since $\partial(X)$ is tight,
$|M \cap \partial(X)|=1$. Consequently, $M_1:=M \cap E(G_1)$
is a \pema\ of $G_1 - V(Q_1)$.

Now suppose that $x \in V(Q_1)$, and let $e$ and $f$ denote the two edges of~$Q_1$ incident with~$x$.
Let $Q_2$ denote an even cycle of~$G_2$ containing $e$ and $f$. (To see why such a cycle exists,
consider the symmetric difference of two perfect matchings of~$G_2$: one containing $e$ and another containing $f$.)
Observe that $Q:=Q_1 \cup Q_2$ is an even cycle in~$G$, and let $M$ denote a \pema\ of~$G-V(Q)$.
Note that $M \cap \partial(X) = \emptyset$ and that $M_1:=M \cap E(G_1)$ is a \pema\ of~$G_1-V(Q_1)$.

In both cases, we have shown that $Q_1$ is a conformal cycle of~$G_1$, whence $G_1$ is \ce.
This completes the proof of
Lemma~\ref{prp:cycle-extendability-inherited-through-tight-cuts}.
\end{proof}

 It is worth noting that the converse of the above lemma does not hold in general. For instance, the graph shown in Figure~\ref{fig:not a cycle-extendable graph} is not \ce~since the cycle shown using a thick line is not conformal; however, its bricks and braces are \ce.  Lemma~\ref{prp:cycle-extendability-inherited-through-tight-cuts} immediately yields the following.

\begin{cor}
\label{Each-brick-brace-ce}
Each brick and brace of a \ce\ graph is also \ce.\qed
\end{cor}

Lov\'azs's theorem (\ref{THE UNIQUE DECOMPOSITION THEOREM}) leads to certain graph invariants that play a crucial role in the theory of \mcg s. In the following two sections, we discuss a couple of these invariants and related concepts that are relevant to our work.

\subsection{Near-bricks versus near-bipartite graphs}
\label{Section-3.3}

It follows from the Unique Tight Cut Decomposition Theorem (\ref{THE UNIQUE DECOMPOSITION THEOREM}) that the number of bricks of a \mcg\ $G$ (obtained by any application of the tight cut decomposition procedure) is an invariant; we denote this by $b(G)$. For instance, $b(G) = 2$ for the graph $G$ shown in Figure~\ref{fig:not a cycle-extendable graph}.

It is worth noting that $b(G)=0$ if and only if $G$ is bipartite. We say that $G$ is a {\em near-brick} if $b(G) = 1$. In particular, every brick is a near-brick. The following is an easy consequence of Theorem \ref{THE UNIQUE DECOMPOSITION THEOREM} that we will find useful in the next section.

\begin{cor}
\sloppy
\label{near-brick-characterization}
    For a near-brick, given any nontrivial tight cut $\partial(X)$, one of the \mbox{$\partial(X)$-contractions} is bipartite and the other one is a near-brick. \qed
\end{cor}

We now observe an easy refinement of the above corollary. Let $G$ be a near-brick that is not a brick and let $\partial(X)$ denote a nontrivial tight cut. By Corollary \ref{near-brick-characterization}, one of the $\partial(X)$-contractions, say $G_1:=G/\overline{X}$, is bipartite. Let $Y$ be a minimal (not necessarily proper) subset of $X$ such that $|Y| \geq 3$ and $D:=\partial_{G_1}(Y)$ is a tight cut of $G_1$. By choice of $Y$, the graph $G_1/\overline{Y}$ is a brace, where $\overline{Y}:= V(G_1) - Y$.  Observe that $D$ is a nontrivial tight cut of~$G$ and that one of the $D$-contractions of $G$ is isomorphic to $G_1/\overline{Y}$. This proves the following well-known fact. 

\begin{cor}
\label{lemma-about-tight-cut-in-near-brick}
    In any near-brick that is not a brick, there exists a nontrivial tight cut $\partial(X)$ such that one of the $\partial(X)$-contractions is a brace and the other one is a near-brick.\qed
\end{cor}

Recall from Section \ref{Section1.1} that a near-bipartite graph is a \mcg, say $G$, that has a removable double ear $R:=\{P_1,P_2\}$ such that $G-R:=G-P_1-P_2$ is bipartite.
The following result of~\cite{clm02b} implies that every near-bipartite graph is a near-brick.
\begin{thm}
\label{Thm:CLM-near-brick}
    Let $G$ be a \mcg, and let $R$ be a removable double ear. Then $b(G-R)=b(G)-1$.
\end{thm}

The following result is an immediate consequence of Theorem \ref{Thm:CLM-near-brick}, and it implies that every near-brick that has a removable double ear is near-bipartite. For an example, see the graph $R_8^{-}$ shown in Figure~\ref{fig:ear-dec-petersen}.
\begin{cor}
\label{cor:near-brick-ver-near-bip}
    Let $G$ be a near-brick, and let $R$ be a removable double ear. Then $G-R$ is bipartite and \mc. \qed
\end{cor}

We remark that a near-brick need not be near-bipartite. For instance, wheels (defined in Section \ref{Section 4.2}), of order six or more, are bricks that are not near-bipartite. This is easily seen using the next proposition --- that follows from the facts: (i) the number of odd faces of a plane graph $G$, denoted by $\sf{f_{odd}}(G)$, is even, and (ii) deleting any edge may reduce $\sf{f_{odd}(G)}$ by at most two.

\begin{prop}
\label{prop:near-bipartite-necessary}
    Every near-bipartite plane graph $G$ satisfies $\sf{f_{odd}}(G) \in \{2,4\}$. \qed
\end{prop}

Lastly, we use $p(G)$ to denote the number of Petersen bricks of a \mcg~$G$ (obtained by any application of the tight cut decomposition procedure), where {\em Petersen brick} refers to the Petersen graph up to multiple edges. In the next section, we apply a result of~\cite{cali08} to deduce that every \ce~graph is either bipartite or otherwise a near-brick whose unique brick is not a Petersen brick.

\subsection{Cycle-extendable\ implies bipartite or near-brick}
\label{reduce to near-brick}
 The aforementioned result of Carvalho and Little pertains to the cycle space $\mathcal{C}(G)$ and its various subspaces that are well-studied vector spaces; see \cite{bomu08}. We let  $\mathcal{C}^{e}(G)$ denote the {\em even space} of a graph $G$ --- that is, the subspace of the cycle space~$\mathcal{C}(G)$ spanned by the even cycles. Likewise, $\mathcal{A}(G)$ denotes the {\em alternating space} --- that is, the subspace of $\mathcal{C}^{e}(G)$ spanned by the conformal cycles.

For $2$-connected graphs, it is well-known that the $\mathcal{C}^e(G)$ is a proper subspace of $\mathcal{C}(G)$ if and only if $G$ is nonbipartite. In the same spirit, the following result of~\cite{cali08} characterizes the \mcg s for which the alternating space $\mathcal{A}(G)$ is a proper subspace of the even space $\mathcal{C}^e(G)$ --- in terms of the invariants $b$ and $p$. 

\begin{thm}
\label{car-lit-result}
    For a \mcg\ $G$, the following are equivalent:
    \begin{enumerate}[(i)]
    \item $\mathcal{A}(G)$ is a proper subspace of $\mathcal{C}^{e}(G)$,
    \item there exists an even cycle $C \in \mathcal{C}^{e}(G) - \mathcal{A}(G)$, and
    \item $b(G) + p(G) > 1$.
    \end{enumerate}
\end{thm}

Observe that, for a \mcg\ $G$, the inequality $b(G)+p(G) > 1$ holds if and only if $G$ is nonbipartite, and either (i) $G$ is not a near-brick or otherwise (ii) $G$ is a near-brick whose unique brick is the Petersen brick. The following is an immediate consequence of the above theorem that we alluded to earlier.

\begin{cor}{\sc[Cycle-Extendable implies Bipartite or Near-Brick]}\newline
\label{near-brick-cor}
Every \ce\ graph is either bipartite or otherwise a near-brick whose unique brick is not the Petersen brick.
\end{cor}
\begin{proof}
    We shall prove the contrapositive. Let $G$ be a \mcg\ such that $b(G) + p(G) > 1$. By Theorem \ref{car-lit-result}, there exists an even cycle $C \in \mathcal{C}^{e}(G) - \mathcal{A}(G)$. Clearly, $C$ is not a conformal cycle; whence $G$ is not \ce.
\end{proof}

We now deduce another consequence of the above corollary and some of the earlier results; in the case of bipartite graphs, we shall find a result of Chartrand, Kaugars and Lick useful. 

\begin{cor}{\sc[Minimum Degree Three or more implies Brick]}\newline
\label{DegreeThree-Graph}
    Every simple planar \ce\ graph, with minimum degree three or more, is a brick.
\end{cor}
\begin{proof}
Let $G$ be a simple planar \ce~graph with $\delta(G) \geq 3$. First, suppose that $G$ is bipartite. Consider a planar embedding of $G$. Since $G$ is 2-connected and $\delta(G) \geq 3$, it follows from a result of~\cite{chartKauLick} that there exists a vertex $v$ so that $G-v$ is a 2-connected (plane) graph. We invoke Whitney's well-known result \cite{bomu08}, and we let $C$ denote the facial (even) cycle that bounds the face containing the point corresponding to the deleted vertex $v$. By planarity, all neighbors of $v$ lie on $C$. Thus, $C$ is not conformal and $G$ is not \ce; contradiction.

Hence, $G$ is nonbipartite. By Corollary \ref{near-brick-cor}, $G$ is a near-brick. Suppose to the contrary that $G$ is not a brick. Using Lemma \ref{lemma-about-tight-cut-in-near-brick}, there exists a nontrivial tight cut $\partial(X)$ so that one of the $\partial(X)$-contractions, say $G_1:=G/\overline{X}$, is a brace. We invoke Corollaries \ref{Each-brick-brace-ce} and \ref{brace-characterize} to infer that $G_1$ is isomorphic to $C_4$ (up to multiple edges). Since $G$ is simple, the multiple edges of $G_1$ (if any) are incident with the contraction vertex. Let $v$ denote the unique vertex of $G_1$ that is not incident with any edge in $\partial(X)$; see Figure~\ref{fig:Illustration of Theorem degree 3}. Observe that $d_{G_1}(v)=2$ and that $d_{G}(v)=2$. This contradicts our hypothesis that $\delta(G) \geq 3$.
\end{proof}

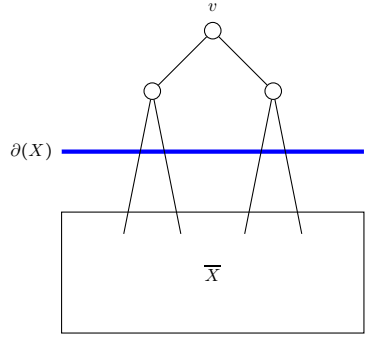
\begin{figure}
    \centering
    \begin{tikzpicture}[every node/.style={draw=black, circle, scale=0.66,fill=white}, scale=0.4,rotate=90]
        \node[label=above:{$v$}] (a1) at (0,0){};
        \node (a2) at (-2,-2){};
        \node (a3) at (-2,2){};
        \draw (a3) -- (a1) -- (a2);
        \draw[ultra thick,blue] (-4,6) -- (-4,-5);
        \node[draw=none] (a4) at (-4,6){$\partial(X)$};
         \draw (-6,5) rectangle (-10,-5);
         \node[draw=none] (a10) at (-8,0){$\overline{X}$};
        \node[draw=none] (a5) at (-7,3){};
        \node[draw=none] (a6) at (-7,1){};
        \node[draw=none] (a7) at (-7,-3){};
        \node[draw=none] (a8) at (-7,-1){};
        \draw (a5) -- (a3) -- (a6);
        \draw (a7) -- (a2) -- (a8);
    \end{tikzpicture}
    \caption{illustration for the proof of Theorem~\ref{DegreeThree-Graph}}
    \label{fig:Illustration of Theorem degree 3}
\end{figure}

Inspired by the proof of the above corollary in the bipartite case, we introduce the following definition that we shall find useful later. For a vertex $v$ of a graph $G$, a cycle $C$ (in $G-v$) is said to be {\em $v$-isolating} if $v$ is an isolated vertex in the graph $G-V(C)$. 

In order to characterize planar \ce~graphs, it follows from Proposition~\ref{prop:irreducible-reduction} and Corollary~\ref{DegreeThree-Graph} that it suffices to characterize (i) planar \ce\ bricks, and (ii)~planar \ce~irreducible graphs that have a vertex of degree two. In the case of planar bricks, we shall solve a more general problem that we discuss in the next section.

\section{\texorpdfstring{$K_{2,3}$}{}-freeness}
\label{sec:k23-free}
To {\em bisubdivide} an edge means to subdivide it by inserting an even number of subdivision vertices.
A graph $H$ is a {\em bisubdivision} of a graph $G$ if $H$ may be obtained from $G$ by selecting any (possibly empty) subset $F \subseteq E(G)$ and bisubdividing each edge in $F$, or equivalently, if $G$ may be obtained from $H$ by a sequence of series reductions.


Given a fixed graph $J$, a graph $G$ is said to be {\em $J$-free} if $G$ does not contain any subgraph~$H$ that is a bisubdivision of $J$; otherwise, we say that $G$ is {\em $J$-based}. The bicorn $R_8$, shown in Figure~\ref{fig:$K_{2,3}$-based graph}, is $K_{2,3}$-based; the thick lines indicate a bisubdivision of $K_{2,3}$. The pentagonal prism minus a specific edge, shown in Figure~\ref{fig:$K_{2,3}$-free graph but not ce}, is $K_{2,3}$-free; the reader may verify this using Lemma~\ref{lem:planar-ce-implies-K23free} and Proposition~\ref{prop:prism ce}.

\begin{figure}[!htb]
    \centering
    \begin{subfigure}{0.5\textwidth}
            \centering
              \begin{tikzpicture}[every node/.style={draw=black, circle,scale=0.66}, scale=0.5]
                   \node (a1) at (2,2){};
                   \node[fill=black] (a2) at (3.5,2){};
                   \node (a3) at (5,2){};
                   \node (a7) at (7,0){};
                   \node (a8) at (0,0){};
                   \node[fill=black] (a4) at (0,4){};
                   \node (a5) at (3.5,4){};
                   \node (a6) at (7,4){};
                   \draw[ultra thick] (a1) -- (a2);
                   \draw[ultra thick] (a2) -- (a3);
                   \draw (a3) -- (a6) -- (a5);
                   \draw[ultra thick] (a5)-- (a4) -- (a1); 
                   \draw[ultra thick] (a4) -- (a8) -- (a7) -- (a3);
                   \draw (a8) -- (a1);
                   \draw (a6) -- (a7);
                   \draw[ultra thick] (a2) -- (a5);
            \end{tikzpicture}
         \caption{}
        \label{fig:$K_{2,3}$-based graph}
    \end{subfigure}%
    \hfill
       \begin{subfigure}{0.5\textwidth}
            \centering
             \begin{tikzpicture}[every node/.style={draw=black, circle, scale=0.66}, scale=0.5]
                \node[draw=none] (a0) at (0:0) {};
                \node (a1) at (30:3) {};
                \node (a2) at (30:5) {};
                \node (a3) at (90:3) {};
                \node (a4) at (90:5) {};
                \node (a5) at (150:3) {};
                \node (a6) at (150:5) {};
                \node (a7) at (215:3) {};
                \node (a8) at (220:5) {};
                \node (a9) at (325:3) {};
                \node (a10) at (320:5) {};

                \draw (a1) -- (a3);
                \draw (a3) -- (a5);
                \draw[ultra thick] (a5) -- (a7) -- (a9) -- (a1);
                \draw (a2) -- (a4) -- (a6);
                \draw[ultra thick] (a6) -- (a8) -- (a10) -- (a2);
                \draw[ultra thick] (a1) -- (a2);
                \draw[ultra thick] (a5) -- (a6);
                \draw (a7) -- (a8);
                \draw (a9) -- (a10);

         \end{tikzpicture}
         \caption{}
        \label{fig:$K_{2,3}$-free graph but not ce}
    \end{subfigure}%
     \caption{(a) the bicorn $R_8$; (b) the pentagonal prism minus one edge}
 \end{figure}
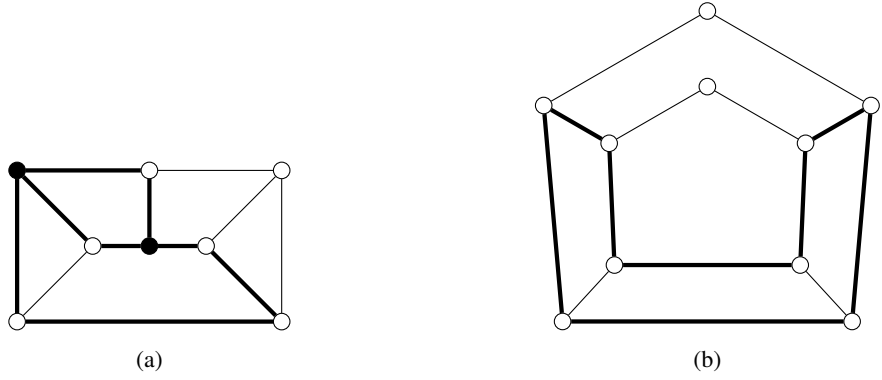
\subsection{Planarity: cycle-extendable implies \texorpdfstring{$K_{2,3}$}{}-free}

                 


We now discuss a necessary condition for a planar graph to be \ce.

\begin{lem}{\sc [Cycle-Extendability implies $K_{2,3}$-freeness in Planar Graphs]}\newline
\label{lem:planar-ce-implies-K23free}
If a planar \mcg\ is \ce\ then it is $K_{2,3}$-free.
\end{lem}
\begin{proof}
We prove the contrapositive.
Let $G$ denote any planar embedding of a planar \mcg\ that is $K_{2,3}$-based, 
and let $H$ denote a subgraph that is a bisubdivision\ of~$K_{2,3}$.
Observe that $H$ comprises two cubic vertices, say~$u$~and~$v$,
and three internally-disjoint even $uv$-paths, say $P_1,P_2$ and $P_3$. See Figure~\ref{fig:bisubdivision-of-K_{2,3}}.
Any two of these paths, say~$P_i$~and~$P_j$ (where $i <j$), comprise a facial cycle of~$H$,
say~$Q_{ij}$. We will argue that at least one of these three cycles is not conformal.

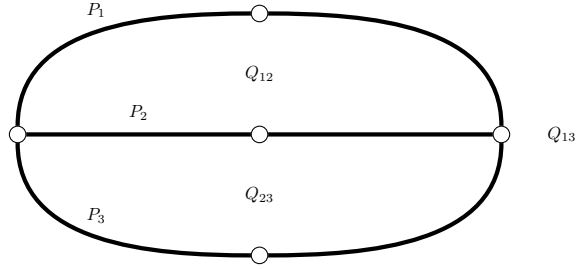
\begin{figure}[!htb]
    \centering
    \begin{tikzpicture}[every node/.style={draw=black, circle,scale=0.66}, scale=0.2]
        \node (a1) at (0,0){};
        \node (a2) at (16,8){};
        \node (a3) at (16,0){};
        \node (a4) at (16,-8){};
        \node (a5) at (32,0){};
        \node[draw=none] (a6) at (16,4){$Q_{12}$};
        \node[draw=none] (a6) at (16,-4){$Q_{23}$};    '
        \node[draw=none] (a6) at (36,0){$Q_{13}$};    '
        \draw[ultra thick] (a1) --node[above,draw=none]{$P_2$} (a3) -- (a5);
        \draw[ultra thick,in=180,out=90] (a1) to node[above,draw=none]{$P_1$} (a2);
        \draw[ultra thick] (a2) to[in=90,out=0] (a5);
        \draw[ultra thick] (a1) to[in=180,out=270]node[above,draw=none]{$P_3$} (a4);
        \draw[ultra thick] (a4) to[in=270,out=0] (a5);
    \end{tikzpicture}
    \caption{an illustration for the proof of Lemma~\ref{lem:planar-ce-implies-K23free}}
    \label{fig:bisubdivision-of-K_{2,3}}
\end{figure}

As per the Jordan Curve Theorem, each cycle $Q_{ij}$ partitions the rest of the plane into two regions:
(i) the interior denoted by ${\sf int}(Q_{ij})$ and (ii) the exterior denoted by ${\sf ext}(Q_{ij})$.
We adjust notation so that ${\sf int}(Q_{ij})$ refers to the region that does not meet the path $P_k$
(where $k$ is distinct from $i$ and $j$) as shown in Figure~\ref{fig:bisubdivision-of-K_{2,3}}.

Now, assume without loss of generality that each of $Q_{12}$ and $Q_{23}$ is conformal.
Since $Q_{12}$ is conformal, ${\sf int}(Q_{12})$ contains an even number of vertices.
Likewise, ${\sf int}(Q_{23})$ contains an even number of vertices.
Observe that ${\sf ext}(Q_{13})$ contains precisely all of the vertices that lie in
${\sf int}(Q_{12}) \cup {\sf int}(Q_{23})$, plus the internal vertices of the path $P_2$.
In particular, ${\sf ext}(Q_{13})$ contains an odd number of vertices; ergo, $Q_{13}$ is not conformal.
Thus $G$ is not \ce. This completes the proof of Lemma \ref{lem:planar-ce-implies-K23free}.
\end{proof}

The above lemma was proved by the second author and appeares in his MSc thesis \cite{paus22}. However, we learnt later this was already observed by~\cite{Zhli12}.

It is worth noting that the converse of the above lemma does not hold in general. Figure~\ref{fig:$K_{2,3}$-free graph but not ce} shows an example of a (nonbipartite) planar \mcg\ that is $K_{2,3}$-free but not \ce; the even cycle depicted by a thick line is not conformal. However, the converse indeed holds in the case of bipartite graphs; this was proved by~\cite{Zhli12}, and independently by the second author \cite{paus22}. This inspires the following decision problem.

\begin{prb}
    Given a \mcg~$G$, decide whether $G$ is $K_{2,3}$-free.
\end{prb}

Thus, in the context of planar nonbipartite \mcg s, \ce\ graphs comprise a proper subset of $K_{2,3}$-free graphs. We are currently working on characterizing $K_{2,3}$-free \mcg s, and have managed to solve the problem in the case of $3$-connected graphs.

In light of Lemma~\ref{lem:planar-ce-implies-K23free}, in order to conclude that a planar graph is not \ce,~it suffices to show the existence of a bisubdivision of $K_{2,3}$. In the following section, we discuss a sufficient condition for the existence of a bisubdivision of $K_{2,3}$ that is applicable to $3$-connected nonbipartite graphs. Since bricks are $3$-connected (see Proposition~\ref{prp:brick-brace-3-connected}) and nonbipartite, this will help us in characterizing planar \ce~bricks. In fact, we will obtain a characterization of planar $K_{2,3}$-free bricks.

\subsection{\texorpdfstring{$3$}{}-connectedness: mixed bicycle implies \texorpdfstring{$K_{2,3}$}{}-based}
A pair of vertex-disjoint simple cycles, say~$(Q_o,Q_e)$, is said to be a {\em mixed bicycle} if $Q_o$ is odd and $Q_e$ is even. (Here, simple just means that the length of $Q_e$ is four or more.) The following is our promised sufficient condition for $3$-connected (nonbipartite) graphs.

\begin{lem}{\sc [Mixed Bicycle implies $K_{2,3}$-based in $3$-connected Graphs]}\newline
\label{lem:3-conn-mixed-bicycle-implies-K23-bisub}
If a $3$-connected graph has a mixed bicycle then it is $K_{2,3}$-based.
\end{lem}
\begin{proof}
Let $(Q_o, Q_e)$ denote a mixed bicycle in a $3$-connected graph~$G$. By Menger's Theorem,
there exist three vertex-disjoint paths --- say $P_i$ (where $i \in \{1,2,3\}$) ---
each of which has one end, say $u_i$, in $Q_o$ and the other end, say $v_i$, in $Q_e$.
Let $Q_1, Q_2$ and $Q_3$ denote the three edge-disjoint paths of $Q_e$ that have both ends in $\{v_1,v_2,v_3\}$.
In particular, $Q_e = Q_1 \cup Q_2 \cup Q_3$.

Since $Q_e$ is even, at least one of $Q_1, Q_2$ and $Q_3$ is even. Adjust notation so that $Q_1$ is even,
and let $v_1$ and $v_2$ denote the ends of $Q_1$. Note that $Q_1$ and $Q_2 \cup Q_3$ are edge-disjoint
even $v_1v_2$-paths. It remains to find one more even $v_1v_2$-path that is
edge-disjoint with $Q_1 \cup (Q_2 \cup Q_3)$. Such a path is easily obtained by combining $P_1 \cup P_2$ with
the $u_1u_2$-path of $Q_o$ of appropriate parity. This completes the proof.
\end{proof}    

We remark that the converse of the above lemma does not hold in general. Clearly, $3$-connected bipartite \mcg s are counterexamples; however, one can also construct nonbipartite counterexamples such as certain subgraphs of $K_5$ and $K_6$. We conclude this section with the following consequence of Lemmas~\ref{lem:3-conn-mixed-bicycle-implies-K23-bisub} and \ref{lem:planar-ce-implies-K23free}.

\begin{cor}
\label{cor:mixed-bicycle-implies-not-cycle-extendable}
If a $3$-connected planar \mcg\ has a mixed bicycle then it is not \ce. \qed
\end{cor}


\section{Planar \texorpdfstring{$K_{2,3}$}{}-free~bricks}
\label{section-4}

Our characterization of planar $K_{2,3}$-free bricks relies on two main ingredients, one of which is Lemma~\ref{lem:3-conn-mixed-bicycle-implies-K23-bisub}. The other one is the brick generation theorem due to~\cite{noth07}, which states that all simple bricks may be constructed from five infinite families and the Petersen graph by means of four operations. We refer to these families as {\em Norine-Thomas families}; furthermore, we refer to their members, as well as to the Petersen graph, as {\em Norine-Thomas bricks}. 

In the next section, we discuss the planar Norine-Thomas families, and we use Lemma~\ref{lem:3-conn-mixed-bicycle-implies-K23-bisub} to classify the $K_{2,3}$-free ones. This classification will serve as the base case in our inductive proof of the characterization of planar $K_{2,3}$-free bricks.


\subsection{Planar Norine-Thomas bricks}
\label{Section 4.2}
We begin this section by describing two of the five Norine-Thomas families --- odd wheels and odd prisms --- whose members happen to be planar as well as \ce. 

A graph obtained from an odd cycle $Q:=u_0u_1\ldots u_{2k}$, by adding a new vertex $h$ and edges $hu_i$ for each $i \in \{0,\ldots,2k\}$, is called an {\em odd wheel}, or simply a {\em wheel}. The vertex $h$ is called its {\em hub} and the cycle $Q$ is called its {\em rim}. Let $G$ be a wheel and let $C$ denote an even cycle. Clearly, $h\in V(C)$ and $G-V(C)$ is an odd path, whence $C$ is conformal. This proves the following.
\begin{prop}
    Wheels are \ce.\qed
\end{prop}

A graph obtained from two vertex-disjoint copies of an odd cycle, say $Q:=w_0w_1\ldots w_{2k}$ and $Q^{'}:=z_0z_1\ldots z_{2k}$, by adding edges $w_iz_i$ for each $i \in \{0,1,\ldots,2k\}$, is called an {\em odd prism}, or simply a {\em prism}. Observe that prisms are {\em cubic} --- that is, the degree of each vertex is precisely three.

In order to prove that prisms are \ce, we will first observe that they are near-bipartite and we shall locate all of their removable double ears; the reader may recall definitions from Section~\ref{Section-3.3}. Before that, we state a simplification of the notion of removable double ears that is applicable to irreducible graphs.

For a \mcg\ $G$, an edge $e$ is {\em removable} if $G-e$ is \mc, and a pair of distinct edges $R:=\{\alpha,\beta\}$ is a {\em removable doubleton} if neither $\alpha$ nor $\beta$ is removable but $G-R$ is \mc. For instance, the graph $R_8^{-}$, shown in Figure~\ref{fig:ear-dec-petersen}, has a removable doubleton depicted by thick lines. Note that, if $G$ is irreducible, each removable double ear is a removable doubleton and vice-versa. Furthermore, if $G$ is near-bipartite then, by Corollary~\ref{cor:near-brick-ver-near-bip}, the graph $G-R$ is bipartite for each removable doubleton $R:=\{\alpha,\beta\}$; also, if $A$ and $B$ denote the color classes of $G-R$ then (up to relabeling) the edge $\alpha$ has both ends in $A$ whereas $\beta$ has both ends in $B$. Using these observations, the following is easily proved; we shall it find very useful in future sections, and we shall invoke it implicitly.

\begin{lem}
\label{lem:4.4}
    For a near-bipartite graph $G$, if $R$ is any removable doubleton and $C$ is any even cycle, then $|C\cap R| \in \{0,2\}$. \qed
\end{lem}

Let $G$ be a prism with vertex-disjoint odd cycles $Q:=w_0w_1\ldots w_{2k}$ and $Q^{'}:=z_0z_1\ldots z_{2k}$ as in the definition stated earlier. Observe that $\{w_iw_{i+1},z_iz_{i+1}\}$ is a removable doubleton of $G$ for each $i \in \{0,1,\ldots,2k\}$, where arithmetic is modulo $2k$. In particular, each vertex is incident with two distinct removable doubletons. Using this observation and Lemma~\ref{lem:4.4}, and the fact that $G$ is cubic, we infer that an even cycle $C$ of $G$ contains $w_i$ if and only if it contains~$z_i$; consequently, $M:=\{w_iz_i | i \in \{0,1, \dots ,2k\} {\rm ~and~ } \{w_i,z_i\} \cap V(C) = \emptyset \}$ is a \pema\ of $G-V(C)$; thus, $C$ is conformal. This proves the following.
\begin{prop}
\label{prop:prism ce}
    Prisms are \ce.\qed
\end{prop}

There are two more Norine-Thomas families each of whose member is planar; these are ``staircases'' and ``truncated biwheels''. The smallest member of each of these families is the triangular prism~$\overline{C_6}$. We refer the reader to \cite{km16} for their descriptions --- using which they may easily verify the following.

\begin{prp}
\label{prp:staircases-truncated-biwheels-K23based}
Every staircase as well as truncated biwheel, except $\overline{C_6}$, has a mixed bicycle; consequently, each of them is $K_{2,3}$-based.
\qed
\end{prp}

To summarize, we have proved the following.
\begin{cor}{\sc[Planar $K_{2,3}$-free Norine-Thomas Bricks]} \newline
\label{summarize-norine-thomas-bricks}
    Wheels and prisms are the only planar Norine-Thomas bricks that are {\mbox {$K_{2,3}$-free}}; in fact, they are also \ce.\qed
\end{cor}


In the next section, we describe the Norine-Thomas brick generation theorem.

\subsection{Norine-Thomas brick generation theorem}

We first state the induction viewpoint of the Norine-Thomas result using the terminology of~\cite{clm15}. 

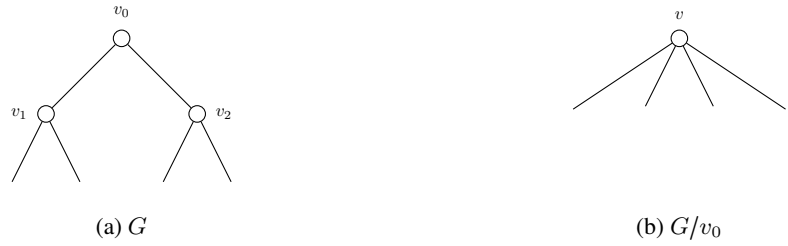
\begin{figure}[!htb]
    \centering
    \begin{subfigure}{0.5\textwidth}
            \centering
            \begin{tikzpicture}[every node/.style={draw=black, circle,scale=0.66}, scale=0.5]
                \node[label=above:{$v_0$}] (a1) at (0,0){};
                \node[label=left:{$v_1$}] (a2) at (-2,-2){};
                \node[label=right:{$v_2$}] (a3) at (2,-2){};
                \node[draw = none] (a4) at (-3,-4){};
                \node[draw = none] (a5) at (-1,-4){};
                \node[draw = none] (a6) at (3,-4){};
                \node[draw = none] (a7) at (1,-4){};
                \draw (a5) -- (a2) -- (a4);
                \draw (a7) -- (a3) -- (a6);
                \draw (a2) -- (a1) -- (a3);
            \end{tikzpicture}
         \caption{$G$}
    \end{subfigure}%
     \begin{subfigure}{0.5\textwidth}
            \centering
            \begin{tikzpicture}[every node/.style={draw=black, circle,scale=0.66}, scale=0.5]
                \node[label=above:{$v$}] (a1) at (0,0){};
                \node[draw = none] (a4) at (-3,-2){};
                \node[draw = none] (a5) at (-1,-2){};
                \node[draw = none] (a6) at (3,-2){};
                \node[draw = none] (a7) at (1,-2){};
                \draw (a5) -- (a1) -- (a4);
                \draw (a7) -- (a1) -- (a6);
                \node[draw = none] (a2) at (-3,-4){};
            \end{tikzpicture}
         \caption{$G/v_0$}
    \end{subfigure}%
    \caption{an illustration of the bicontraction operation}
    \label{fig:bicontraction operation}
\end{figure}

Let $G$ be a \mcg, and let $v_0$ be a vertex of degree two that has two distinct neighbors $v_1$ and $v_2$. The {\em bicontraction} of $v_0$ is the operation of contracting the two edges $v_0v_1$ and $v_0v_2$, and we use $G/v_0$ to denote the resulting graph; see Figure~\ref{fig:bicontraction operation}. Note that $\partial(X)$ is a tight cut, where $X:=\{v_0,v_1,v_2\}$, and $G/v_0$ is simply the $\partial(X)$-contraction $G/X$, as defined in Section~\ref{Section:tight-cut}; consequently, $G/v_0$ is \mc. It is worth noting that even if $G$ is simple, $G/v_0$ need not be simple. 
The following is easily proved and we shall use its contrapositive in our proof of the characterization of planar $K_{2,3}$-free bricks. 

\begin{lem}
\label{lem:bi-contraction-preserves-J-bi-subdivisions}
Let $G$ be a \mcg~and let $v_0$ denote a vertex of degree two that has two distinct neighbors.
 For any graph~$K$ with maximum degree $\Delta(K) \leq 3$,
if $G/v_0$ is $K$-based then $G$ is also $K$-based. \qed
\end{lem}

Now, let simple $G$ be a brick and $e$ denote a removable edge as defined in the preceding section. Since each vertex of $G$ has three or more distinct neighbors, the \mcg\ $G-e$ is irreducible\ and it has zero, one or two vertices of degree two. The \mcg~$J$ obtained from $G-e$ by bicontracting each of its vertices of degree two (if any) is referred to as the {\em retract of $G-e$}. Note that its minimum degree $\delta(J) \geq 3$; however, $J$ need not be a brick. For instance, in the Petersen graph $\mathbb{P}$, shown in Figure~\ref{fig:the petersen graph}, each edge is removable; however, the reader may verify that $b(J) = 2$, where $J$ is retract of $\mathbb{P}-e$ for any edge $e$.
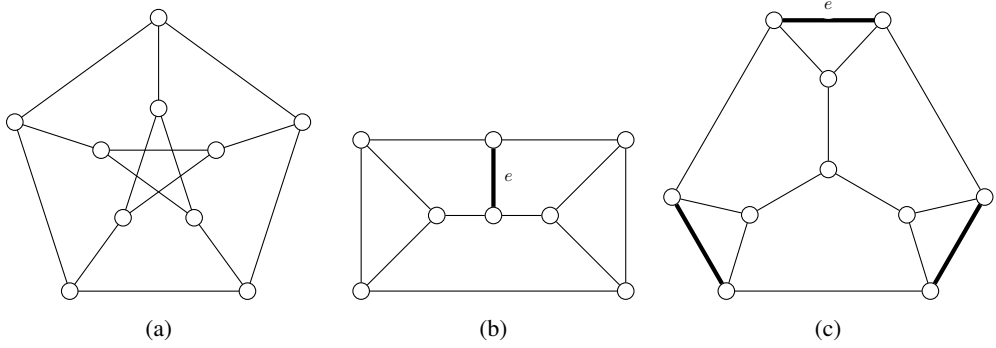
\begin{figure}[!htb]
    \centering
    \begin{subfigure}{0.3\textwidth}
            \centering
            \begin{tikzpicture}[every node/.style={draw=black, circle, scale=0.66,fill=white}, scale=0.4]
                 \foreach \x in {0,72,...,288}
                    {
                        \pgfmathsetmacro\index{\x/72+6}
                        \begin{pgfonlayer}{background}
                            \draw (\x+18:2) -- (\x+2*72+18:2);
                            \draw (\x+18:5) -- (\x+72+18:5);
                            \draw (\x+18:2) -- (\x+18:5);
                        \end{pgfonlayer}
                            \node at (\x+3*72+18:2) {};
                        \pgfmathsetmacro\nindex{5-\x/72}
                            \node at (\x+3*72+18:5) {};
                    }
            \end{tikzpicture}
         \caption{}
        \label{fig:the petersen graph}
    \end{subfigure}%
    \begin{subfigure}{0.3\textwidth}
        \centering
         \begin{tikzpicture}[every node/.style={draw=black, circle,scale=0.66}, scale=0.5]
               \node (a1) at (2,2){};
               \node (a2) at (3.5,2){};
               \node (a3) at (5,2){};
               \node (a7) at (7,0){};
               \node (a8) at (0,0){};
               \node (a4) at (0,4){};
               \node (a5) at (3.5,4){};
               \node (a6) at (7,4){};
               \draw (a1) -- (a2) -- (a3) -- (a6) -- (a5) -- (a4) -- (a1); 
               \draw (a4) -- (a8) -- (a7) -- (a3);
               \draw (a8) -- (a1);
               \draw (a6) -- (a7);
               \draw[ultra thick] (a2) --node[right,draw=none]{$e$} (a5);
    \end{tikzpicture}    
        \caption{}
        \label{fig:R_8}
    \end{subfigure}%
    \begin{subfigure}{0.3\textwidth}
        \centering
         \begin{tikzpicture}[every node/.style={draw=black, circle,scale=0.66,fill=white}, scale=0.3,rotate=180]
                \node (a0) at (0:0) {};
                \node (a1) at (30:4) {};
                \node (a8) at (50:7) {};
                \node (a9) at (10:7) {};
                \draw (a1) -- (a8);
                \draw[ultra thick] (a8)-- (a9);
                \draw (a9) -- (a1);

                \node (a3) at (90:-4) {};
                \node (a4) at (70:-7) {};
                \node (a5) at (110:-7) {};
                \draw (a3) -- (a4) -- (a5) -- (a3);
                \draw[ultra thick] (a4) --node[above,draw=none]{$e$} (a5);
                
                \node (a2) at (330:-4) {};
                \node (a6) at (310:-7) {};
                \node (a7) at (350:-7) {};
                \draw (a2) -- (a6);
                \draw[ultra thick] (a6)-- (a7);
                \draw (a7) -- (a2);

                \draw (a9) -- (a5);
                \draw (a4) -- (a7);
                \draw (a6) -- (a8); 
                
                \draw (a1) -- (a0) -- (a2);
                \draw (a0) -- (a3);
         \end{tikzpicture}    
        \caption{}
        \label{fig:tricorn}
    \end{subfigure}%
 \caption{(a) the Petersen graph $\mathbb{P}$; (b) the bicorn $R_8$; (c) the tricorn $R_{10}$.}
\end{figure}

In light of the above discussion, a removable edge $e$ of a brick $G$ is a {\em thin edge} if the retract of $G-e$ is also a brick. For instance, the bicorn $R_8$, shown in Figure~\ref{fig:R_8}, has a unique removable edge $e$ that also happens to be thin --- since the retract of $R_8-e$ is the brick $K_4$ with multiple edges. \cite{clm06} proved that every brick, except $K_4$,  triangular prism~$\overline{C_6}$ and the Petersen graph, has a thin edge.

A thin edge $e$ of a simple brick $G$ is a {\em strictly thin edge} if the retract of $G-e$ is also simple. For instance, the tricorn~$R_{10}$, shown in Figure~\ref{fig:tricorn}, has precisely three strictly thin edges; for each such edge, say $e$, the retract of $R_{10}-e$ is precisely the wheel $W_5$. On the other hand, the bicorn $R_8$ has no strictly thin edges.
We are now ready to state the generation theorem for simple bricks due to~\cite{noth07}.

\begin{thm}{\sc [Strictly Thin Edge Theorem]}\newline
\label{Strictly Thin Edge Theorem}
    Every simple brick, except for the Norine-Thomas bricks, has a strictly thin edge. 
\end{thm}

In order to prove our characterization of planar $K_{2,3}$-free bricks, we require the generation viewpoint of the above theorem which states that every simple brick may be generated from a Norine-Thomas brick by a sequence of four ``expansion operations'' as described by~\cite{clm06}. 

Let $H$ be a graph that has a vertex of degree at least four, say $v$. Consider any graph $G$ that is obtained from $H$ by replacing the vertex $v$ by two new vertices $v_1$ and $v_2$, distributing the edges of $\partial_{H}(v)$ amongst $v_1$ and $v_2$ so that each of them receives at least two, and then adding a new vertex $v_0$ as well as edges $v_0v_1$ and $v_0v_2$. We say that $G$ is obtained from $H$ by {\em bisplitting} the vertex $v$. Observe that $H=G/v_0$; see Figure~\ref{fig:bicontraction operation}. It is easily seen that $G$ is matching covered if and only if $H$ is matching covered. 

Each of the aforementioned four ``expansion operations'' consists of bisplitting zero, one or two vertices of a simple brick $J$ and then adding a suitable edge $e$ in order to obtain a larger simple brick $G$. Conversely, $e$ is a strictly thin edge of $G$ and $J$ is the retract of $G-e$. The interested reader may refer to~\cite{clm06} for the exact descriptions of these expansion operations. However, we shall only use the bisplitting operation in our proof of the characterization of planar $K_{2,3}$-free bricks; in particular, we will find the following observation useful.

\begin{cor}
\label{lem:bisplitting-preserves-mixed-bicycle}
Let $G$ denote a graph obtained from a graph~$H$ by bisplitting a vertex of degree four or more.
If $H$ contains a mixed bicycle then so does $G$. \qed
\end{cor}

The above can also be deduced from Lemma~\ref{lem:bi-contraction-preserves-J-bi-subdivisions} by noting that a mixed bicycle is simply a bisubdivision of the disjoint union of $C_{3}$ and $ C_{4}$, and that bicontraction may be viewed as the ``inverse'' of bisplitting.

\subsection{Characterization of planar \texorpdfstring{$K_{2,3}$}{}-free bricks}
We are now ready to present our characterization of planar $K_{2,3}$-free bricks.

\begin{thm}{\sc [Planar $K_{2,3}$-free Bricks]}\newline
\label{thm:simple-planar-K23free-bricks}
A simple planar brick is $K_{2,3}$-free if and only if it is either a wheel or a prism.
\end{thm}
\begin{proof}
In light of Corollary \ref{summarize-norine-thomas-bricks},
we only need to prove the forward direction.
To this end, let $G$ denote a simple planar brick that is $K_{2,3}$-free. We proceed by induction on
the number of edges.

If $G$ is a \NT\ brick then we are done by Corollary~\ref{summarize-norine-thomas-bricks}.
Otherwise, by Theorem~\ref{Strictly Thin Edge Theorem}, 
$G$ has a strictly thin edge, say~$e$, and we let $J$ denote the retract of $G-e$. Since $G$ is $K_{2,3}$-free, so
is $G-e$; by Lemma~\ref{lem:bi-contraction-preserves-J-bi-subdivisions}, so is $J$. Ergo,
$J$ is a simple planar brick that is $K_{2,3}$-free; by the induction hypothesis, $J$ is either a wheel
or a prism. In particular, $G$ is obtained from $J$ by zero, one or two bisplitting operations followed by adding an edge. We shall consider two cases depending on the number of bisplitting operations --- either zero or at least one.

First, suppose that $G$ is obtained from $J$ by adding an edge; that is, $G=J + e$. Since $J$ is either a wheel or a prism, the reader may easily verify that if $|V(J)| \geq 8$ then $G$ contains a mixed bicycle;
by Lemma~\ref{lem:3-conn-mixed-bicycle-implies-K23-bisub}, $G$ is $K_{2,3}$-based; contradiction. 
On the other hand, if $|V(J)| = 6$, then $G$ is obtained from $W_5$ or $\overline{C_6}$ by adding the edge $e$;
in each case, one may observe that $K_{2,3}$ is a subgraph of $G$; contradiction.

Now suppose that $G$ is obtained from $J$ by one or two bisplitting operations followed by adding an edge. In particular, $J$ has a vertex of degree four or more. Ergo, $J$ is an odd wheel
of order six or more.
Furthermore, the first bisplitting operation must bisplit the hub~$h$ of~$J$, the reader may easily verify
that the resulting graph has a mixed bicycle.
By Lemma~\ref{lem:bisplitting-preserves-mixed-bicycle}, $G$ contains a mixed bicycle.
By Lemma~\ref{lem:3-conn-mixed-bicycle-implies-K23-bisub}, $G$ is $K_{2,3}$-based; contradiction. This completes the proof of Theorem~\ref{thm:simple-planar-K23free-bricks}.
\end{proof}

The above theorem, along with Corollary~\ref{summarize-norine-thomas-bricks}, yields the following.
\begin{cor}{\sc [Planar Cycle-Extendable Bricks]}\newline
\label{cor:Planar Cycle-Extendable Bricks}
A simple planar brick is \ce\ if and only if it is either a wheel or a prism. \qed
\end{cor}

We now proceed towards our final goal of characterizing planar \ce~{\mbox irreducible} graphs.

\section{Planar \ce~irreducible graphs}
\label{section-5}
In light of Proposition~\ref{prop:irreducible-reduction}, Corollary~\ref{DegreeThree-Graph} and Theorem~\ref{thm:simple-planar-K23free-bricks}, it remains to characterize those planar \ce~irreducible graphs that have a vertex of degree two; let $G$ be such a graph and let $x_0$ denote any vertex of degree two. In our inductive proof of the Main Theorem (\ref{thm:characterize-nonbipartite-ce-graphs}), we shall consider the smaller graph $J^{'}:=G/x_0$. Observe that $J^{'}$ is a planar \mcg. By Lemma~\ref{prp:cycle-extendability-inherited-through-tight-cuts}, $J^{'}$ is also \ce. However, $J^{'}$ need not be irreducible; in fact, $J'$ need not be simple. We now introduce the notion of an ``osculating bicycle''  that will help us in deducing that $J^{'}$ is either simple, or otherwise has exactly two multiple edges.

\subsection{Osculating bicycle}
\label{osculating-bicycle}

\indent A pair of cycles $(Q,Q^{'})$ in a graph $G$ is said to be an {\em osculating bicycle} if (i) they intersect in precisely one vertex, and (ii) they have the same parities. Furthermore, if each of them is an odd cycle then we refer to $(Q,Q^{'})$ as an {\em odd} osculating bicycle. Likewise, if each of them is an even cycle then $(Q,Q^{'})$ is an {\em even} osculating bicycle. 

We recall the following definition from Section~\ref{reduce to near-brick}. For a vertex $v$ of a graph $G$, a cycle~$C$ (in $G-v$) is $v$-isolating if $v$ is an isolated vertex in the graph $G-V(C)$. The next lemma provides an easy sufficient condition to deduce that a graph (obtained via bisplitting a vertex) is not \ce. We shall find this lemma immensely useful in our proof of the Main Theorem (\ref{thm:characterize-nonbipartite-ce-graphs}).
\begin{lem}{\sc [Osculating Bicycle Lemma]}\newline
\label{lem:osculating-bicycle}
    Let $G$ be a graph and $x_0$ denote a vertex of degree two that has two distinct neighbors. If $J^{'} := G/x_0$ has an  osculating bicycle $(Q,Q^{'})$ such that neither~$Q$ nor~$Q^{'}$ is a cycle in $G$, then $Q \cup Q^{'}$ is an $x_0$-isolating even cycle in $G$.
\end{lem}
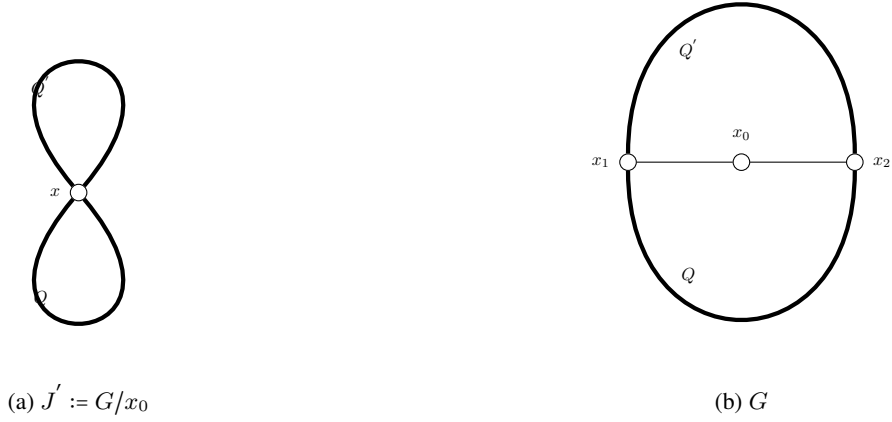
\begin{figure}[!htb]
    \centering
    \begin{subfigure}{0.5\textwidth}
            \centering
                \begin{tikzpicture}[every node/.style={draw=black, circle, scale=0.66}, scale=0.5]
    	        \node[label=left:{$x$}] (a2) at (0,0){} ;  
                    \node [draw=none] (a4) at (-1,-2.8){$Q$};
                    \node [draw=none] (a5) at (-1,2.8){$Q^{'}$};
                    \draw [in=230, out=310, looseness=50.25,ultra thick] (a2) to (a2);
                    \draw [in=130, out=50, looseness=50.25,ultra thick] (a2) to (a2);
    	
                \end{tikzpicture}
             
         \caption{$J^{'}:=G/x_0$}
        \label{fig:J=G/x_0}
    \end{subfigure}%
    \hfill
    \begin{subfigure}{0.3\textwidth}
        \centering
            \begin{tikzpicture}[every node/.style={draw=black, circle, scale=0.66}, scale=0.5]
                \node[label=above:{$x_0$}] (a1) at (3,0){} ;  
                \node[label=left:{$x_1$}] (a2) at (0,0){} ;  
                \node[label=right:{$x_2$}] (a3) at (6,0){} ;  
                \node [draw=none] (a4) at (1.6,3){$Q^{'}$};
                \node [draw=none] (a5) at (1.6,-3){$Q$};
                \draw (a3) -- (a1) -- (a2);
                \draw [in=90, out=90, looseness=2.25,ultra thick] (a3) to (a2);
                \draw [in=270, out=270, looseness=2.25,ultra thick] (a3) to (a2);
         \end{tikzpicture}

        \caption{$G$}
        \label{fig:graphG}
     \end{subfigure}
     \hfill
     \caption{an illustration for the Osculating Bicycle Lemma and its proof}
     \label{fig:osculatingbicyclelemma}
 \end{figure}
\begin{proof}
We let $x_1$ and $x_2$ denote the two distinct neighbors of $x_0$ in $G$, and let $x$ denote the contraction vertex in $J^{'}:=G/x_0$. Suppose that $(Q,Q^{'})$ is an osculating bicycle in $J^{'}:=G/x_0$ such that neither~$Q$ nor~$Q^{'}$ is a cycle in $G$. Since $Q$ is not a cycle in $G$, the vertex $x$ belongs to $Q$, and one of the two edges of $Q \cap \partial_{J^{'}}(x)$ is incident with $x_1$ in $G$ and the other edge is incident with $x_2$ in $G$; see Figure~\ref{fig:osculatingbicyclelemma}. An analogous statement holds for $Q^{'}$. Since $(Q,Q^{'})$ is an osculating bicycle in $J^{'}$, we infer that $Q$ and $Q^{'}$ meet precisely at the vertex $x$ in $J^{'}$. This fact, combined with the fact that their parities are the same, implies that $Q \cup Q^{'}$ is an $x_0$-isolating even cycle in $G$, as shown in Figure~\ref{fig:graphG}.
\end{proof}

For distinct vertices $u$ and $v$ of a graph $G$, we let $\mu_G(u,v)$ denote the number of edges joining $u$ and $v$. We shall find this notation useful in the proof of the following corollary of the above lemma.

\begin{cor}
\label{cor:J is 2-stable and parallel edge}
    Let $G$ be a simple \ce\ graph, $x_0$ be a vertex of degree two and let~$J^{'} := G/x_0$. Then either $J^{'}$ is simple or otherwise $J^{'}$ has precisely two multiple edges. Furthermore, if $G$ is irreducible then the underlying simple graph $J$ of $J^{'}$ is also irreducible.
\end{cor}
\begin{proof}
    Observe that, since $G$ is simple, all multiple edges in $J^{'}$ (if any) are incident with the contraction vertex that we denote by $x$.
    
    To prove the first part, suppose to the contrary that $J^{'}$ has more than two multiple edges; we consider two cases. First suppose that there exists a vertex $u \in V(J^{'})-x$ such that $\mu_{J^{'}}(x,u) \geq 3$. Since $J^{'}=G/x_0$, it follows from the pigeonhole principle that $G$ is not simple; contradiction. Otherwise, there exist distinct $u_1,u_2 \in V(J^{'})-x$ such that $\mu_{J^{'}}(x,u_1) = \mu_{J^{'}}(x,u_2) = 2$. Let $e$ and $e^{'}$ denote distinct edges joining $u_1$ and $x$; likewise, let $f$ and $f^{'}$ denote distinct edges joining $u_2$ and $x$. Observe that $(Q:=ee^{'},Q^{'}:=ff^{'})$ is an osculating bicycle in $J^{'}$ and neither $Q$ nor $Q^{'}$ is a cycle in the simple graph $G$. By the Osculating Bicycle Lemma~(\ref{lem:osculating-bicycle}), $Q \cup Q^{'}$ is an $x_0$-isolating $4$-cycle in $G$; this contradicts the hypothesis that $G$ is \ce. This proves the first part.

    Now suppose that $G$ is irreducible and let $x_1$ and $x_2$ denote the neighbors of $x_0$; thus, $d_G(x_1) \geq 3$ and $d_G(x_2) \geq 3$. Consequently, $d_{J^{'}}(x) \geq 4$ and degree two vertices of $J^{'}$ comprise a stable set. If $J^{'} = J$ then we are done. Otherwise, by the first part, there exists a unique vertex $u \in V(J^{'}) - x$ such that  $\mu_{J^{'}}(x,u) = 2$; let $e$~and~$e^{'}$ denote the two edges joining $x$ and $u$. Adjust notation so that $J= J^{'} - e^{'}$. Since $G$ is simple, adjust notation so that $e \in \partial_G(x_1)$ and $e^{'} \in \partial_G(x_2)$. Suppose to the contrary that, unlike $J^{'}$, the vertices of degree two of $J$ do not comprise a stable set. Since $d_{J}(x) \geq 3$, we infer that $d_{J}(u) = 2$ and that the neighbour of $u$, distinct from $x$, also has degree two in $J$. We choose $f \in \partial_G(x_1)-e-x_1x_0$ and $f^{'} \in \partial_G(x_2)-e^{'}-x_2x_0$. Since $f,f^{'} \in \partial_J(x)$, we let $Q$ denote a conformal cycle containing $f$ and $f^{'}$ in $J$. Since $d_J(u) = 2$ and $e=ux$, we infer that $u\notin V(Q)$. Thus $(Q,Q^{'}:=ee^{'})$ is an osculating bicycle in $J^{'}$ and neither $Q$ nor $Q^{'}$ is a cycle in $G$. By the Osculating Bicycle Lemma~(\ref{lem:osculating-bicycle}), $Q \cup Q^{'}$ is an $x_0$-isolating cycle in $G$. This contradicts the hypothesis that $G$ is \ce, and proves the second part.
\end{proof}

We now proceed to describe an important class of bipartite \ce~graphs that will play an important role in our description as well as appreciation of planar \ce~irreducible~graphs.

\subsection{Half biwheels}
 A {\em\bw} is any graph $H$ obtained from an even path $P$[$A,B$] with ends, say $u$ and $v$ in $A$, by introducing a new vertex, say $h$, and joining $h$ with each vertex in $A$. Figure~\ref{ear-decom-K_33} shows \bw s of orders six and eight. The ends of $P$ are called the {\em corners}, and $h$ is called the {\em hub}, of $H$. Observe that any \bw\ is a \ce\ bipartite \mcg. Also, we allow $P$ to be $K_1$; in this case, $H$ is the smallest \bw~$K_2$, either vertex may be regarded as the hub $h$ and the other one as the (only) corner $u=v$. Note that $C_4$ is the second smallest \bw; any vertex may be regarded as the hub $h$ and its two distinct neighbors as the corners $u$~and~$v$. In all other cases, the hub and (distinct) corners are uniquely determined. The reader may easily verify the following.

\begin{prop}
\label{prop:corner-hub-path-conformal}
    Every \bw~$H$ is bipartite and \ce. Furthermore, each path of $H$, starting at the hub and ending at a corner, is conformal.\qed
\end{prop}

In each of the following four sections, we describe a family of nonbipartite planar \ce\ irreducible graphs $\mathcal{G}_0,\mathcal{G}_1,\mathcal{G}_2$ and $\mathcal{G}_3$. Each member of these families is constructed from $k \geq 1$ \bw (s) --- by perhaps identifying all of the hubs (only in the case of $\mathcal{G}_1$) --- and introducing additional vertices (only in the case of $\mathcal{G}_3$) as well as edges. For the sake of brevity, we shall find it convenient to use arithmetic modulo $k$. The reader may verify that these families are pairwise disjoint.

Additionally, the reader may verify that each member of these families is a subdivision of a $3$-connected planar graph. Hence, by Whitney's Theorem \cite[Theorem 10.28]{bomu08}, it admits a unique planar embedding. We will provide planar embeddings of certain small members of these families. Some of our propositions shall refer to the cyclic ordering of edges (incident at a common vertex) induced by the planar embedding.

For each family, we will state propositions pertaining to the existence of specific even cycles and osculating bicycles. We shall find these very useful in our proof of the Main Theorem in order to invoke the Osculating Bicycle Lemma (\ref{lem:osculating-bicycle}) and thus significantly reduce the amount of case analysis required. In particular, propositions pertaining to even cycles shall be useful in the case where $x$ is a cubic vertex of $J$, as per the notation in the proof of Lemma~\ref{lem:osculating-bicycle} and Figure~\ref{fig:osculatingbicyclelemma}, whereas propositions pertaining to osculating bicycles shall be invoked in the case where $x$ has degree four or more in $J$.

\subsection{Generalized prisms \texorpdfstring{$\mathcal{G}_{0}$}{}}
In this section, we introduce our first family of graphs, {\em generalized prisms}, denoted by~$\mathcal{G}_0$. Each member of $\mathcal{G}_0$, say $J$, is constructed from the union of $k$ disjoint \bw s, say $H_i$[$A_i,B_i$] for $i \in \{0,1,\ldots,k-1\}$, where $k$ is odd and at least three, each labeled as follows --- if the hub $h_i$ belongs to $A_i$ then label $h_i$ as $w_i=x_i$ and the corners as $y_i$ and $z_i$, and if the hub $h_i$ belongs to $B_i$ then label $h_i$ as $y_i=z_i$ and the corners as $w_i$ and $x_i$ --- by adding the following edges: $\alpha_i:=x_iw_{i+1}$ and $\beta_i:=y_iz_{i+1}$ for each $0 \leq i \leq k-1$. Figure~\ref{fig:a member of G_0} depicts an example where the \bw s~$H_i$ are shown in blue.

\begin{figure}[!htb]
        \begin{tikzpicture}[every node/.style={draw=black, circle,scale=0.5}, scale=0.4]
          
            \node[label=below:{\Large $z_0$}] (a2) at (1.5,4) {};
            \node[label=below:{$ $}] (a3) at (3,4) {};
            \node[label=below:{$ $}] (a4) at (4.5,4) {};
            \node[label=below:{$ $}] (a5) at (6,4) {};
            \node[label=below:{\Large $y_0$}] (a8) at (7.5,4) {};
            \node[label=above:{\Large $w_0=x_0$}] (a11) at (4.5,8) {};
            \draw[blue] (a2) -- (a3) -- (a4);
            \draw[blue] (a4) -- (a5) -- (a8);
            \draw[blue] (a11) -- (a2);
            \draw[blue] (a11) -- (a4);
            \draw[blue] (a11) -- (a8);

            \node[label=above:{\Large $w_1$}] (a21) at (10,8) {};
            \node[label=above:{}] (a22) at (11.5,8) {};
            \node[label=above:{\Large $x_1$}] (a23) at (13,8) {};
            \node[label=below:{\Large $y_1=z_1$}] (a24) at (11.5,4) {};
            \draw[blue] (a21) -- (a22) -- (a23) -- (a24) -- (a21);

            \node[label=below:{\Large $y_2=z_2$}] (a31) at (15.5,4) {};
            \node[label=above:{\Large $w_2=x_2$}] (a32) at (15.5,8) {};
           \draw[blue] (a31)--(a32);

            \node[label=above:{\Large $w_3$}] (a42) at (18,8) {};
            \node[label=below:{$ $}] (a43) at (19.5,8) {};
            \node[label=below:{$ $}] (a44) at (21,8) {};
            \node[label=below:{}] (a45) at (22.5,8) {};
            \node[label=above:{\Large $x_3$}] (a48) at (24,8) {};
            \node[label=below:{\Large $y_3=z_3$}] (a411) at (21,4) {};
            \draw[blue] (a42) -- (a43) -- (a44);
            \draw[blue] (a44) -- (a45) -- (a48);
            \draw[blue] (a411) -- (a42);
            \draw[blue] (a411) -- (a44);
            \draw[blue] (a411) -- (a48);

              \node[label=below:{\Large $y_4=z_4$}] (a51) at (26.5,4) {};
            \node[label=above:{\Large $w_4=x_4$}] (a52) at (26.5,8) {};
         
           \draw[blue] (a51)--(a52);

           \draw[ thick] (a8)  -- (a24);
           \draw[ thick] (a11)  -- (a21);

           \draw[ thick] (a24) -- (a31);
           \draw[ thick] (a31) -- (a411);
           \draw[ thick] (a411) -- (a51);
           \draw[ thick] (a51) to[in=205,out=335] (a2);

           \draw[ thick] (a23) -- (a32);
           \draw[ thick] (a32) -- (a42);
           \draw[ thick] (a48) -- (a52);
           \draw[ thick] (a52) to[in=150,out=30,looseness=1.1] (a11);
           
    \end{tikzpicture}
    \caption{a generalized prism}
    \label{fig:a member of G_0}
\end{figure}
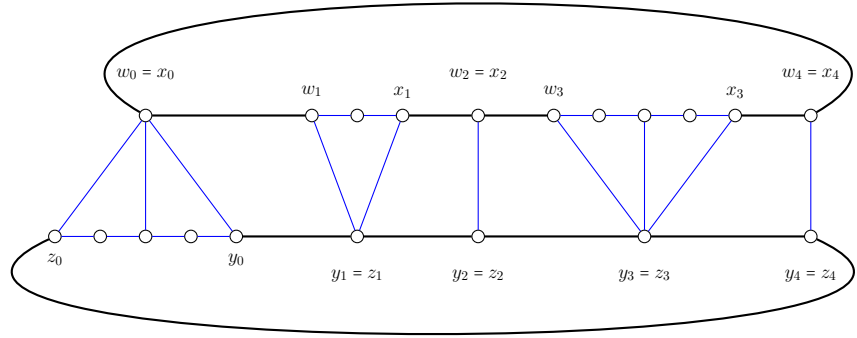

Each member $J$ of $\mathcal{G}_0$ is a near-bipartite planar (matching covered) graph. Furthermore, $R_i:=\{\alpha_i,\beta_i\}$ denotes a removable doubleton for each $0\leq i \leq k-1$, and these are the only removable doubletons. The components of $J-R_0-R_1-\ldots -R_{k-1}$ are precisely the \bw s $H_0,H_1,\ldots,H_k$. Observe that if each \bw\ $H_i$ is isomorphic to $K_2$ then $J$ is simply a prism. In this sense, the family $\mathcal{G}_0$ indeed generalizes 
the prisms.  In order to prove that members of $\mathcal{G}_0$ are \ce, we shall find the following lemma useful; the reader may prove it easily.

\begin{lem}
\label{lem:conformality-intrasubgraph}
    Let $J$ be a graph, let $(V_0,V_1,\ldots,V_{k-1})$ denote a partition of $V(J)$, let $H_i$ denote the induced subgraph $J[V_i]$ for each $0 \leq i \leq k-1$, and let $L$ denote a subgraph of $J$. If $L_i:=L \cap H_i$ is a conformal subgraph of $H_i$ for each $0 \leq i \leq k-1$, then $L$~is a conformal subgraph of $J$. \qed
\end{lem}

We are now ready to prove that generalized prisms are \ce.

\begin{prop}
\label{prop:G_0-ce}
     Each member of $ \mathcal{G}_{0}$ is \ce.
\end{prop}
\begin{proof}
    Let $J\in \mathcal{G}_0$ and let $R_0,R_1,\ldots,R_{k-1}$ denote its removable doubletons. Let $H_0,H_1,\ldots,H_{k-1}$ denote the (\bw) components of $J-R_0-R_1-\ldots-R_{k-1}$; adjust notation so that $\partial(V(H_i)) = R_{i-1} \cup R_i$ and let $V_i:=V(H_i)$ for each $0\leq i \leq k-1$. 
    
    Let $C$ denote an even cycle of $J$. Let us pick $i \in \{0,1,\ldots,k-1\}$, and make some observations.  By Lemma~\ref{lem:4.4}, $|C \cap R_i| \in \{0,2\}$. Since $\partial(V(H_i)) = R_{i-1} \cup R_i$, we infer that $C \cap \partial(V(H_i)) \in \{\emptyset,R_{i-1},R_i,R_{i-1}\cup R_{i}\}$. Let $C_i := C \cap H_i$. Note that $C \cap \partial(V(H_i)) =\emptyset$  if and only if either $C_i = C$ or $C_i$ is the null subgraph of $H_i$. Secondly, $C \cap \partial(V(H_i)) \in\{R_{i-1},R_i\}$  if and only if $C_i$ is a path of the \bw~$H_i$ that starts at the hub and ends at a corner. Lastly, $C \cap \partial(V(H_i)) = R_{i-1} \cup R_i$ if and only if $C_i$ is a spanning subgraph of $H_i$. By Proposition~\ref{prop:corner-hub-path-conformal}, $C_i$~is a conformal subgraph of $H_i$. It follows immediately from Lemma~\ref{lem:conformality-intrasubgraph}, with $C$ playing the role of $L$, that $C$ is a conformal cycle of $J$. Thus $J$ is \ce.
\end{proof}

The following property of generalized prisms, pertaining to the existence of even cycles containing a specified cubic vertex but not containing a specified neighbor, is easily verified.

\begin{prop}{\sc [Even Cycles in Generalized Prisms]}\newline
\label{prop:G_0_cubic_vertex}
    Let $J \in \mathcal{G}_0$ and let $R_0,R_1,\ldots,R_{k-1}$ denote its removable doubletons.
    Let $x$ denote a cubic vertex and let $H$ denote the (\bw) component of $J-R_0-R_1-\ldots -R_{k-1}$ that contains~$x$. Let $w$ denote any neighbor of $x$ in~$J$ so that (i) either $w \notin V(H)$ or (ii)~$w \in V(H)$ and $d_{J}(w)=2$. Then $J-w$ has an even cycle that contains $x$. 
\end{prop}
\begin{proof}
    We adopt notation from the definition of generalized prisms, and adjust notation so that $H=H_1$. We leave it as an exercise for the reader to locate the desired even cycle in the induced subgraph $J[V(H_0) \cup V(H_1) \cup V(H_2)]$ --- by considering various cases depending on whether $H_1$ is $K_2$ or not, and depending on the specific choices of $x$ and $w$.
\end{proof} 

For any member $J \in \mathcal{G}_0$, it is easy to see that the planar embedding of $J$ has precisely two odd faces whose boundaries are vertex-disjoint and comprise a spanning subgraph of $J$. We use $D$ and $D^{'}$ to denote these facial cycles. By adjusting notation, one of these cycles contains each $\alpha_i$ whereas the other one contains each $\beta_i$, where  \{$R_i:=\{\alpha_i,\beta_i\},0\leq i \leq k-1\}$ is the set of all removable doubletons. Using these observation, we now proceed to discuss the existence of specific osculating bicycles in generalized prisms.

\begin{figure}[!htb]
    \centering
         \begin{tikzpicture}[every node/.style={draw=black, circle,scale=0.5}, scale=0.35]
          
            \node[label=below:{\large $z_0$}] (a2) at (-3,4) {};
            \node[label=below:{$ $}] (a3) at (-1.5,4) {};
            \node[label=below:{$ $}] (a4) at (0,4) {};
            \node[label=below:{$ $}] (a5) at (1.5,4) {};
            \node[label=below:{\large $y_0$}] (a8) at (3,4) {};
            \node[label=above:{\large $w_0=x_0$}] (a11) at (0,8) {};
            \draw[blue,ultra thick] (a2) -- (a3) -- (a4);
            \draw[blue,ultra thick] (a4) -- (a5) -- (a8);
            \draw (a11) -- (a2);
            \draw (a11) -- (a4);
            \draw (a11) -- (a8);

            \node[label=above:{\large $w_1=x_1$}] (a21) at (9,8) {};
            \node[label=below:{\large $z_1$}] (a22) at (4.5,4) {};
            \node[label=above:{}] (a23) at (6,4) {};
            \node[label=below:{\large $t$}] (a24) at (7.5,4) {};
            \node[label=above:{}] (a25) at (9,4) {};
            \node[label=below:{\large $t^{'}$}] (a26) at (10.5,4) {};
            \node[label=above:{}] (a27) at (12,4) {};
            \node[label=below:{\large $y_1$}] (a28) at (13.5,4) {};
            \draw (a21) -- (a22);
            \draw[blue,ultra thick] (a21) --node[right,draw=none,black]{\large $f$} (a24);
            \draw[blue,ultra thick] (a21) --node[right,draw=none,black]{\large $f^{'}$} (a26);
            \draw (a21) -- (a28);
            \draw[blue,ultra thick] (a22) -- (a23) -- (a24);
            \draw (a24)-- (a25) -- (a26);
            \draw[blue,ultra thick] (a26) -- (a27) -- (a28);

            \node[label=below:{\large $y_2=z_2$}] (a31) at (15.5,4) {};
            \node[label=above:{\large $w_2=x_2$}] (a32) at (15.5,8) {};
           \draw (a31)--(a32);

            \node[label=above:{\large $w_3$}] (a42) at (18,8) {};
            \node[label=below:{$ $}] (a43) at (19.5,8) {};
            \node[label=below:{$ $}] (a44) at (21,8) {};
            \node[label=below:{}] (a45) at (22.5,8) {};
            \node[label=above:{\large $x_3$}] (a48) at (24,8) {};
            \node[label=below:{\large $y_3=z_3$}] (a411) at (21,4) {};
            \draw[red,ultra thick] (a42) -- (a43) -- (a44);
            \draw[red,ultra thick] (a44) -- (a45) -- (a48);
            \draw (a411) -- (a42);
            \draw (a411) -- (a44);
            \draw (a411) -- (a48);

              \node[label=below:{\large $y_4=z_4$}] (a51) at (26.5,4) {};
            \node[label=above:{\large $w_4=x_4$}] (a52) at (26.5,8) {};
         
           \draw (a51)--(a52);

           \draw[blue,ultra thick] (a8)  -- (a22);
           \draw[red,ultra thick] (a11)  --node[above,draw=none,black]{\large $\alpha_0$} (a21);

           \draw[blue,ultra thick] (a28) -- (a31);
           \draw[blue,ultra thick] (a31) -- (a411);
           \draw[blue,ultra thick] (a411) -- (a51);
           \draw[blue,ultra thick] (a51) to[in=205,out=335] (a2);

           \draw[red,ultra thick] (a21) --node[above,draw=none,black]{\large $\alpha_1$} (a32);
           \draw[red,ultra thick] (a32) -- (a42);
           \draw[red,ultra thick] (a48) -- (a52);
           \draw[red,ultra thick] (a52) to[in=150,out=30] (a11);

    \end{tikzpicture}
    \caption{Illustration for the proof of Proposition~\ref{prop:g_0_osculating_bicycle} (i)}
    \label{fig:Illustation of Propositions statement (i)}
\end{figure}
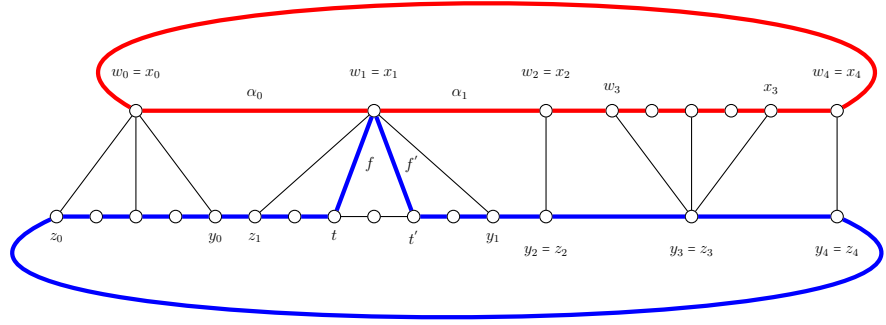

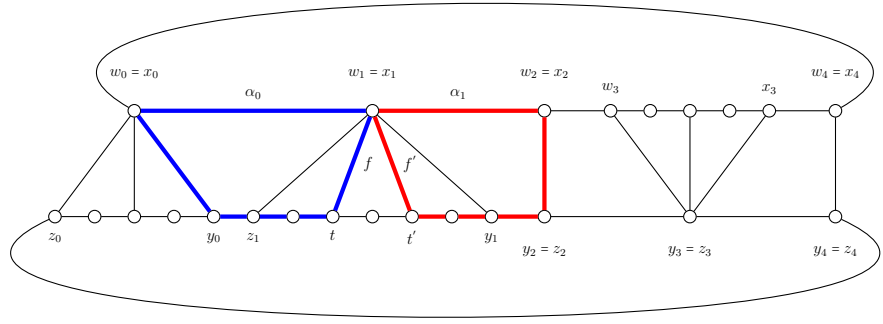
\begin{figure}[!htb]
    \centering
         \begin{tikzpicture}[every node/.style={draw=black, circle,scale=0.5}, scale=0.35]
          
            \node[label=below:{\large $z_0$}] (a2) at (-3,4) {};
            \node[label=below:{\large $ $}] (a3) at (-1.5,4) {};
            \node[label=below:{$ $}] (a4) at (0,4) {};
            \node[label=below:{$ $}] (a5) at (1.5,4) {};
            \node[label=below:{\large $y_0$}] (a8) at (3,4) {};
            \node[label=above:{\large $w_0=x_0$}] (a11) at (0,8) {};
            \draw (a2) -- (a3) -- (a4);
            \draw (a4) -- (a5) -- (a8);
            \draw (a11) -- (a2);
            \draw (a11) -- (a4);
            \draw[blue,ultra thick] (a11) -- (a8);

            \node[label=above:{\large $w_1=x_1$}] (a21) at (9,8) {};
            \node[label=below:{\large $z_1$}] (a22) at (4.5,4) {};
            \node[label=above:{}] (a23) at (6,4) {};
            \node[label=below:{\large $t$}] (a24) at (7.5,4) {};
            \node[label=above:{}] (a25) at (9,4) {};
            \node[label=below:{\large $t^{'}$}] (a26) at (10.5,4) {};
            \node[label=above:{}] (a27) at (12,4) {};
            \node[label=below:{\large $y_1$}] (a28) at (13.5,4) {};
            \draw (a21) -- (a22);
            \draw[blue,ultra thick] (a21) --node[right,draw=none,black]{\large $f$} (a24);
            \draw[red,ultra thick] (a21) --node[right,draw=none,black]{\large $f^{'}$} (a26);
            \draw (a21) -- (a28);
            \draw[blue,ultra thick] (a22) -- (a23) -- (a24);
            \draw (a24)-- (a25) -- (a26);
            \draw[red,ultra thick] (a26) -- (a27) -- (a28);

            \node[label=below:{\large $y_2=z_2$}] (a31) at (15.5,4) {};
            \node[label=above:{\large $w_2=x_2$}] (a32) at (15.5,8) {};
           \draw[red,ultra thick] (a31)--(a32);

            \node[label=above:{\large $w_3$}] (a42) at (18,8) {};
            \node[label=below:{\large $ $}] (a43) at (19.5,8) {};
            \node[label=below:{$ $}] (a44) at (21,8) {};
            \node[label=below:{}] (a45) at (22.5,8) {};
            \node[label=above:{\large $x_3$}] (a48) at (24,8) {};
            \node[label=below:{\large $y_3=z_3$}] (a411) at (21,4) {};
            \draw (a42) -- (a43) -- (a44);
            \draw (a44) -- (a45) -- (a48);
            \draw (a411) -- (a42);
            \draw (a411) -- (a44);
            \draw (a411) -- (a48);

              \node[label=below:{\large $y_4=z_4$}] (a51) at (26.5,4) {};
            \node[label=above:{\large $w_4=x_4$}] (a52) at (26.5,8) {};
         
           \draw (a51)--(a52);

           \draw[blue,ultra thick] (a8)  -- (a22);
           \draw[blue,ultra thick] (a11)  --node[above,draw=none,black]{\large $\alpha_0$} (a21);

           \draw[red,ultra thick] (a28) -- (a31);
           \draw (a31) -- (a411);
           \draw (a411) -- (a51);
           \draw (a51) to[in=205,out=335] (a2);

           \draw[red,ultra thick] (a21) --node[above,draw=none,black]{\large $\alpha_1$} (a32);
           \draw (a32) -- (a42);
           \draw (a48) -- (a52);
           \draw (a52) to[in=150,out=30] (a11);

    \end{tikzpicture}
    \caption{Illustration for the proof of Proposition~\ref{prop:g_0_osculating_bicycle} (ii)}
    \label{fig:Illustation of Propositions statement (ii)}
\end{figure}
\begin{prop}{\sc [Osculating Bicycles in Generalized Prisms]}\newline
\label{prop:g_0_osculating_bicycle}
    Let $J \in \mathcal{G}_0$ and let $x$ denote a vertex of degree four or more. Let $\alpha_0$ and~$\alpha_1$ denote removable doubleton edges in $\partial(x)$ and let $f, f^{'} \in \partial(x) - \{\alpha_0, \alpha_1\}$ so that $\alpha_0, \alpha_1, f^{'},f$ appear in this cyclic order in the planar embedding of $J$. Then there exist:
    \begin{enumerate}[(i)]
        \item an odd osculating bicycle $(Q,Q^{'})$ such that $\alpha_0, \alpha_1 \in E(Q)$ and $f,f^{'} \in E(Q^{'})$, and
        \item an even osculating bicycle $(C,C^{'})$ such that $\alpha_0,f \in E(C)$ and $\alpha_1,f^{'} \in E(C^{'})$.
    \end{enumerate} 
\end{prop}
\begin{proof}
    We adopt notation from the definition of generalized prisms; thus $x=h_1$. Let $f:=h_1t$ and $f^{'}:=h_1t^{'}$. Let $D$ and $D^{'}$ denote the odd facial cycles as discussed earlier, and adjust notation so that $h_1 \in V(D)$.

    To prove (i), we display an odd osculating bicycle $(Q,Q^{'})$ whose each constituent cycle uses precisely one edge from each of the removable doubletons of $J$ as follows: $Q:=D$ and $Q^{'}:=D^{'}+f+f^{'}-t(H_1-h)t^{'}$; see Figure~\ref{fig:Illustation of Propositions statement (i)}.

    To prove (ii), we display an even osculating bicycle $(C,C^{'})$ whose each constituent cycle uses either zero or two edges from each of the removable doubletons of $J$ as follows: $C:=h_1ft(H_1-h)z_1\beta_0y_0x_0\alpha_0h_1$ and $C^{'}:=h_1f^{'}t^{'}(H_1-h)y_1\beta_1z_2w_2\alpha_1h_1$; see Figure~\ref{fig:Illustation of Propositions statement (ii)}.
\end{proof}

In each of the following three sections, we will define a family of graphs $\mathcal{G}_i$ where \mbox{$i \in \{1,2,3\}$}. Each member of $\mathcal{G}_i$ will be defined using some number of \bw s. For each such \bw\ $H_{}[A,B]$, we adopt the convention that $h$ denotes the hub and that $u$~and~$v$ denote the corners (that are not necessarily distinct). Also, this convention is extended naturally to accommodate the use of subscripts.  

\subsection{Generalized wheels \texorpdfstring{$\mathcal{G}_1$}{} }



We describe a family of graphs $\mathcal{G}_1$. Each member of $\mathcal{G}_1$, say~$J$, is constructed from $k$ disjoint \bw s, say $H_i$ for $0 \leq i \leq k-1$, where $k$ is odd and at least three,  by identifying all of their hubs into a single vertex $h$ and adding the set of edges $E_3:=\{v_iu_{i+1}:0\leq i \leq k-1\}$. Figure~\ref{fig:a member of g_1} depicts two examples where the \bw s $H_i$ are shown in blue. We refer to $h$ as a {\em hub of $J$}.

Observe that $J-E_3$ is precisely the (bipartite) graph obtained by identifying the hubs of $k=|E_3|$ \bw s and $h$ is its unique cut vertex. Furthermore, when $J$ is not $K_4$, the set $E_3$ comprises precisely those edges whose both ends are cubic. If each \bw\ $H_i$ (in the above definition) is isomorphic to $K_2$, then $J$ is simply a wheel. In this sense, the family $\mathcal{G}_1$ indeed generalizes wheels. Note that the hub $h$ is the unique vertex of degree four or more except when $J$ is $K_4$.

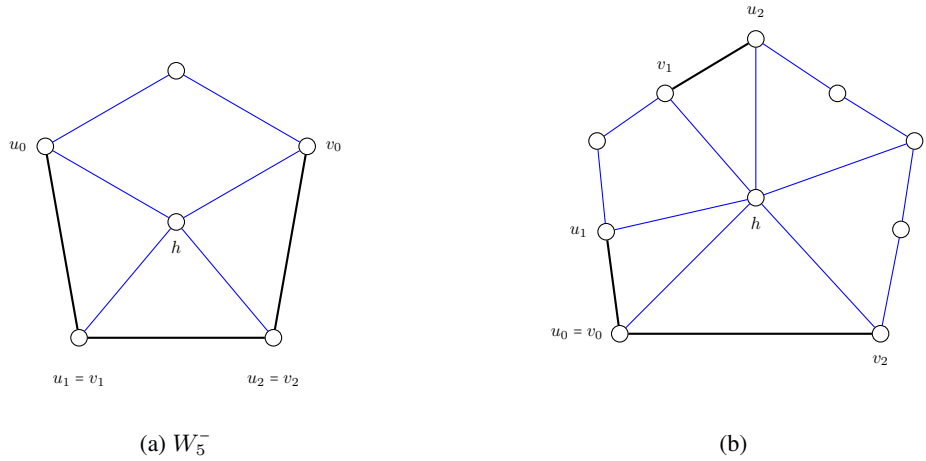
\begin{figure}[!htb]
    \centering
    \begin{subfigure}{0.5\textwidth}
        \centering
      \begin{tikzpicture}[every node/.style={draw=black, circle,scale=0.66}, scale=0.5,rotate=270]
                
                

                \node[label=below:{$h$}] (a1) at (0:0) {};
                \node (a2) at (180:4){};
                \node[label=below:{$u_1=v_1$}] (a3) at (140:-4){};
                \node[label=below:{$u_2=v_2$}] (a4) at (220:-4){};
                \node[label=right:{$v_0$}] (a5) at (120:4){};
                \node[label=left:{$u_0$}] (a6) at (240:4){};
                \draw[thick] (a3) -- (a4);\draw[thick] (a4) -- (a5);
                \draw[blue] (a5) -- (a2); \draw[blue] (a2) -- (a6); \draw[thick] (a6)-- (a3);
                \draw[blue] (a1) -- (a4);
                \draw[blue] (a1) -- (a3);
                \draw[blue] (a1) -- (a5);
                \draw[blue] (a1) -- (a6);
                
    \end{tikzpicture}
     \caption{$W_5^{-}$}
        \label{fig:w_5-}
    \end{subfigure}%
      \begin{subfigure}{0.5\textwidth}
        \centering
    \begin{tikzpicture}[every node/.style={draw=black, circle,scale=0.66}, scale=0.6]
                \node[label=left:{$u_0=v_0$}] (a6) at (0,1.5) {};
                \node (a7) at (-0.5,5.75) {};
                \node[label=left:{$u_1$}] (a3) at (-0.3,3.75) {};
                \node[label=above:{$u_2$}] (a8) at (3,8) {};
                \node[label=above:{$v_1$}] (a5) at (1,6.8) {};
                \node[] (a4) at (4.8,6.8) {};
                \node (a9) at (6.5,5.75) {};
                \node (a2) at (6.2,3.8) {};
                \node[label=below:{$v_2$}] (a10) at (5.75,1.5) {};
                
                \node[label=below:{$h$}] (a1) at (3,4.5){};
                \node[draw=none] (a0) at (-0.1,-0.1){};
                
                \draw[blue] (a1) -- (a6);
                \draw[blue] (a1) -- (a8);
                \draw[blue] (a1) -- (a9);
                \draw[blue] (a1) -- (a10);
                \draw[blue] (a1) -- (a3);
                \draw[blue] (a1) -- (a5);
                \draw[thick] (a6) -- (a3);
                \draw[blue] (a7) -- (a5);
                \draw[thick] (a5) -- (a8);
                \draw[blue] (a7) -- (a3);
                \draw[blue] (a8) -- (a4);
                \draw[blue] (a9) -- (a4);
                \draw[blue] (a2) -- (a10);
                \draw[blue] (a9) -- (a2);
                \draw[thick] (a10) -- (a6);
    \end{tikzpicture}
     \caption{}
        
    \end{subfigure}%
    \caption{generalized wheels}
    \label{fig:a member of g_1}
\end{figure}

Each member $J$ of $\mathcal{G}_1$ is a planar \mcg~and ${\sf f_{odd}}(J)=k+1$. Consequently, they are not near-bipartite in general; see Proposition~\ref{prop:near-bipartite-necessary}. The graph $J$ is near-bipartite if and only if $k=3$ and at least one of $H_0$, $H_1$ and $H_2$ is isomorphic to~$K_2$. Furthermore, if $k=3$ and if $H_i$ denotes a \bw\ that is isomorphic to~$K_2$ then $R_i:= \{u_ih,v_{i+1}u_{i+2}\}$ is a removable doubleton; it follows that the number of \bw s (among $H_0$, $H_1$, and $H_2$) that are isomorphic to $K_2$ equals the number of removable doubletons in $J$. 

To see that each member $J$ of $\mathcal{G}_1$ is \ce, note that $J-h$ is an odd cycle; consequently, if $C$ is any even cycle, then $h \in V(C)$ and $J-V(C)$ is an odd path. 

\begin{prop}
\label{prop:G_1-ce}
   Each member of $\mathcal{G}_1$ is \ce.\qed
\end{prop}
 



We now make an observation pertaining to the existence of even cycles in generalized wheels containing a specified cubic vertex but not containing a specified neighbor.

\begin{prop}{\sc [Even Cycles in Generalized Wheels]}\newline
\label{prop:G_1_cubic_vertex}
    Let $h$ denote the hub of $J \in \mathcal{G}_1 - K_4$ and let $E_3$ denote the set of edges whose both ends are cubic. Let $x$ denote a cubic vertex of $J$ and let $w \neq h$ denote a neighbor of $x$ in $J$. If there is no even cycle in $J-w$ that contains $x$, then (i) $|E_3|=3$ and (ii) each of $x$ and $w$ is an isolated vertex in $J-E_3-h$.
\end{prop}
\begin{proof}
    We adopt notation from the definition of generalized wheels, and adjust notation so that $x \in V(H_1)$ and $w \in V(H_0) \cup V(H_1)$. Assume that there is no even cycle in $J-w$ that contains $x$.

    First, consider the case in which $w \in V(H_1)$. If $x \notin \{u_1,v_1\}$, then the reader may locate an even cycle in $H_1-w$ that contains $x$; a contradiction. Otherwise, we may adjust notation so that $x = u_1$. Observe that $Q:=hu_1v_0(H_0-h)u_0v_{k-1}h$ is an even cycle in $J-w$ that contains~$x$; a contradiction.

    Next, consider the case in which $w \in V(H_0)$; thus $x=u_1$ and $w=v_0$. If $H_1$ is not $K_2$, then $H_1$ has an even cycle containing $x$; this contradicts our assumption. Now suppose that $H_1$ is $K_2$; in other words, $x=u_1=v_1$.  Observe that $Q:=hv_1u_2(H_2-h)v_2u_3h$ is an even cycle that contains $x$; furthermore, $Q$ does not contain $w=v_0$ if and only if $u_3 \neq v_0$. Hence, by our assumption, $u_3 = v_0$. It follows from the definition of generalized wheels that conditions (i)~and~(ii) hold.
\end{proof}

Finally, we proceed to discuss the existence of certain osculating bicycles in generalized wheels.

\begin{figure}[!htb]
    \centering

   \begin{tikzpicture}[every node/.style={draw=black, circle,scale=0.66}, scale=0.6]
        \node[label=above:{$h$}] (a0) at (0,0){};
        \node[label=right:{$u_0=v_0$}] (a1) at (27:3) {};
        \node[label=right:{$u_1$}] (a2) at (54:3) {};
        \node (a3) at (81:3) {};
        \node[label=left:{$v_1$}] (a4) at (108:3) {};
        \node[label=left:{$u_2$}] (a5) at (135:3) {};
        \node (a6) at (162:3) {};
        \node[label=left:{$v_2$}] (a7) at (189:3) {};
        
        \node[label=left:{$u_3$}] (a8) at (216:3) {};
        \node (a9) at (243:3) {};
        \node (a10) at (270:3) {};
        \node (a11) at (297:3) {};   
        \node[label=right:{$v_3$}] (a12) at (324:3) {};   
        \node[label=right:{$u_4 = v_4$}] (a13) at (351:3) {};   

        \draw[ultra thick, blue] (a0) --node[right,draw=none,black]{$e_0$} (a1);
        \draw[ultra thick, blue] (a0) --node[right,draw=none,black]{$e_1$} (a4);
        \draw[ultra thick, blue] (a1) -- (a2) -- (a3) -- (a4);

        \draw[ultra thick, red] (a0) --node[above,draw=none,black]{$e_r$} (a7);
        \draw[ultra thick, red] (a0) --node[right,draw=none,black]{$e_{r+1}$} (a10);
        \draw[ultra thick, red] (a7) -- (a8) -- (a9) -- (a10);

        \draw (a4) -- (a5) -- (a6) -- (a7);
        \draw (a10) -- (a11) -- (a12) -- (a13) -- (a1);

        \draw (a0) -- (a2);
        \draw (a0) -- (a5);
        \draw (a0) -- (a8);
        \draw (a0) -- (a12);
        \draw (a0) -- (a13);
    \end{tikzpicture}

    \caption{illustration for the proof of Proposition~\ref{prop:G_1_non_cubic_vertex} (iii)}
    \label{fig:Illustraton of proposition g_1 noncubic}
\end{figure}
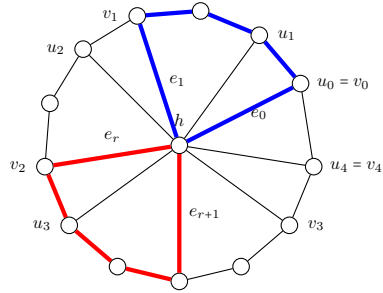

\begin{prop}{\sc [Osculating Bicycles in Generalized Wheels]}\newline
\label{prop:G_1_non_cubic_vertex}
    Let $h$ denote the hub of $J \in \mathcal{G}_1 - K_4$ and let $E_3$ denote the set of edges whose both ends are cubic. Let $H_0,H_1,\ldots,H_k$ denote the (\bw) blocks of $J-E_3$ so that they appear in this cyclic order (around the hub $h$) in the planar embedding of $J$, and adjust notation so that $E_3=\{v_iu_{i+1}: v_i \in V(H_i)~{\rm and }~u_{i+1} \in V(H_{i+1})\}$. The following statements hold:
    \begin{enumerate}[(i)]
        \item for any $r$ such that $0\leq r < k$, if the four edges $hu_0,hv_r,hu_{r+1}$ and $hv_k$ are pairwise distinct then $(Q:=hu_0v_kh,Q^{'}:=hv_ru_{r+1}h)$ is an odd osculating bicycle in $J$.
        \item for any $0 \leq i < r \leq k$, for distinct $e$ and $f$ in $E(H_i) \cap \partial_J(h)$ and distinct $e^{'}$~and~$f^{'}$ in $E(H_r) \cap \partial_J(h)$, there exists an even osculating bicycle $(Q,Q^{'})$ in $J$ such that $e,f \in E(Q)$ and $e^{'},f^{'} \in E(Q^{'})$.
        \item for any $0 < r \leq k$, for any four pairwise distinct edges $e_0,e_1,e_r,e_{r+1} \in \partial_J(h)$ that appear in this cyclic order in the planar embedding of $J$, such that $e_i \in E(H_i)$, there exists an odd osculating bicycle $(Q,Q^{'})$ in $J$ such that $e_0,e_1 \in E(Q)$ and $e_r,e_{r+1} \in E(Q^{'})$.
    \end{enumerate}
\end{prop}
\begin{proof} The reader may easily verify statement (i), whereas statement~(ii) follows from the fact that any two adjacent edges (in $H_i$; likewise, in $H_r$) belong to an even cycle.

We now display an odd osculating bicycle $(Q,Q^{'})$ for statement (iii). For each $i \in \{0,1,r,r+1\}$, we let $w_i$ denote the end of $e_i$ that is distinct from $h$. We define them as follows: $Q:=hw_0(H_0-h)v_0u_1(H_1-h)w_1h$ and $Q^{'}:=hw_r(H_r-h)v_ru_{r+1}(H_{r+1}-h)w_{r+1}h$; see Figure~\ref{fig:Illustraton of proposition g_1 noncubic}. Both cycles use precisely one edge from the set $E_3$; thus, they are odd cycles.
\end{proof}

\subsection{Double \bw s \texorpdfstring{$\mathcal{G}_{2}$}{}}
\label{sec:Double-half-biwheel}




This section describes another family of graphs denoted as $\mathcal{G}_2$. Each member of $\mathcal{G}_2$, say $J$, is constructed from the disjoint union of two \bw s neither of which is isomorphic to~$K_2$, say $H_0$ and $H_1$, by adding the following four edges:
$\alpha_0:=h_1h_0$, $\beta_0:=v_1v_0$, $\alpha_1:=u_0h_1$, and  $\beta_1:=h_0u_1$. Figure~\ref{fig:a member of g_2} depicts an example where the \bw s $H_0$ and $H_1$ are shown in blue. (The reader may verify that if any of $H_0$ and $H_1$ is allowed to be isomorphic to $K_2$, then $J$ is a generalized wheel.)

\begin{figure}[!htb]
    \centering
    \begin{tikzpicture}[every node/.style={draw=black, circle,scale=0.66}, scale=0.45]
           
            \node[label=below:{$h_0$}] (a1) at (5.5,0) {};
            \node[label=above:{$v_0$}] (a2) at (2.5,4) {};
            \node[label=above:{$ $}] (a3) at (4,4) {};
            \node[label=above:{$ $}] (a4) at (5.5,4) {};
            \node[label=above:{$ $}] (a5) at (7,4) {};
            \node[label=above:{$u_0$}] (a6) at (8.5,4) {};
            \node[label=above:{$h_1$}] (a7) at (13.5,4) {};
            \node[label=below:{$u_1$}] (a8) at (12,0) {};
            \node[label=below:{$ $}] (a9) at (13.5,0) {};
            \node[label=below:{$ $}] (a10) at (15,0) {};
            \node[label=below:{$ $}] (a11) at (16.5,0) {};
            \node[label=below:{$ $}] (a12) at (18,0) {};
            \node[label=below:{$ $}] (a13) at (19.5,0) {};
            \node[label=below:{$v_1$}] (a14) at (21,0) {};
            
            \draw[thick] (a8) -- (a1);
            \draw[thick] (a1) -- (a7);
            \draw[thick] (a7) -- (a6);
            \draw[thick] (a2) to[in=230,out=250] (a14);
            \draw[blue] (a1) -- (a2);
            \draw[blue] (a2) -- (a3) -- (a4) -- (a5) -- (a6) -- (a1) -- (a4);
            \draw[blue] (a7) -- (a8) -- (a9) -- (a10);
            \draw[blue] (a10) -- (a11) -- (a12) -- (a13) -- (a14);
            \draw[blue] (a12) -- (a7);
            \draw[blue] (a7) -- (a10) ;
            \draw[blue] (a7) -- (a14);
    \end{tikzpicture}
    \caption{a double \bw}
    \label{fig:a member of g_2}
\end{figure}

The reader may observe that, in the above definition, the \bw s $H_0$~and~$H_1$ are interchangeable and that $h_0$ and $h_1$ are the only vertices of degree four or more. Each member of $\mathcal{G}_2$, say $J$ is a near-bipartite planar (matching covered) graph. Furthermore, $R_0=\{\alpha_0,\beta_0\}$ and $R_1=\{\alpha_1,\beta_1\}$ are the only removable doubletons, and the components of $J-R_0-R_1$ are precisely the \bw s $H_0$ and $H_1$. Using ideas similar to the ones in the proof of Proposition~\ref{prop:G_0-ce}, the reader may verify the following.


\begin{prop}
\label{prop:G_2-ce}
    Each member of $ \mathcal{G}_{2}$ is \ce.\qed
\end{prop}
    

 As usual, we make an observation concerning the existence of even cycles in double \bw s containing a specified cubic vertex but not containing a specified neighbor. We leave its proof as an exercise for the reader.

\begin{prop}{\sc [Even Cycles in Double Half Biwheels]}\newline
\label{prop:G_2_cubic_vertex}
    Let $J \in \mathcal{G}_2$ and let $R_0,R_1$ denote the removable doubletons of~$J$.
    Let $x$ denote a cubic vertex and let $H_1$ denote the (\bw) component of $J-R_0-R_1$ that contains $x$. Let $w$ denote a neighbor of $x$ in~$J$ so that (i) either $w \notin V(H_1)$ or (ii)~$w \in V(H_1)$ and $d_{J}(w)=2$. Then, $J-w$ has an even cycle that contains $x$.\qed
\end{prop}

We now proceed to state two observations pertaining to the existence of osculating bicycles; the first of these deals with vertices of degree four or more.

\begin{prop}{\sc [Osculating Bicycles in Double Half Biwheels~-~I]}\newline
\label{prop:g_2_osculating_bicycle_1}
    Let $J \in \mathcal{G}_2$ and let $x$ denote a vertex of degree four or more. Let $\alpha_0$ and~$\alpha_1$ denote removable doubleton edges in $\partial(x)$ and let $f, f^{'} \in \partial(x) - \{\alpha_0, \alpha_1\}$ so that $\alpha_0, \alpha_1, f^{'},f$ appear in this cyclic order in the planar embedding of $J$. Then there exists an odd osculating bicycle $(Q,Q^{'})$ such that $\alpha_0,f \in E(Q)$ and $\alpha_1,f^{'} \in E(Q^{'})$.
\end{prop}
\begin{proof}
We adopt notation from the definition of double \bw s; thus $x=h_1$. Let $f:=h_1w$ and $f^{'}:=h_1w^{'}$.
We display an odd osculating bicycle $(Q,Q^{'})$ whose each constituent cycle uses precisely one edge from each of the two removable doubletons of $J$ as follows: $Q:=h_1fw(H_1 - h_1)u_1\beta_1h_0\alpha_0h_1$ and $Q^{'}:=h_1f^{'}w^{'}(H_1-h_1)v_1\beta_0v_0(H_0-h_0)u_0\alpha_1h_1$. Figure~\ref{fig:Illustation of Propositions- G_2a} shows an illustration.
\end{proof}

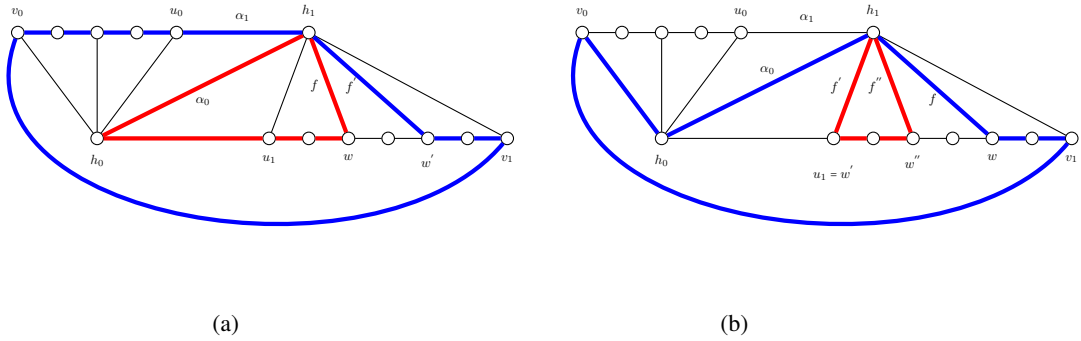
\begin{figure}[!htb]
    \hspace{-20px}
     \begin{subfigure}{0.5\textwidth}
        \centering
        
        \begin{tikzpicture}[every node/.style={draw=black, circle,scale=0.5}, scale=0.35]
               
                 \node[label=below:{$h_0$}] (a1) at (5.5,0) {};
                \node[label=above:{$v_0$}] (a2) at (2.5,4) {};
                \node[label=above:{$ $}] (a3) at (4,4) {};
                \node[label=above:{$ $}] (a4) at (5.5,4) {};
                \node[label=above:{$ $}] (a5) at (7,4) {};
                \node[label=above:{$u_0$}] (a6) at (8.5,4) {};
                \node[label=above:{$h_1$}] (a7) at (13.5,4) {};
                \node[label=below:{$u_1$}] (a8) at (12,0) {};
                \node[label=below:{$ $}] (a9) at (13.5,0) {};
                \node[label=below:{$w$}] (a10) at (15,0) {};
                \node[label=below:{$ $}] (a11) at (16.5,0) {};
                \node[label=below:{$w^{'}$}] (a12) at (18,0) {};
                \node[label=below:{$ $}] (a13) at (19.5,0) {};
                \node[label=below:{$v_1$}] (a14) at (21,0) {};
                \draw[ultra thick,red] (a8) -- (a1);
                \draw[ultra thick,red] (a1) --node[below,draw=none,black]{$\alpha_0$} (a7);
                \draw[ultra thick,blue] (a7) --node[above,draw=none,black]{$\alpha_1$} (a6);
                \draw[ultra thick,blue] (a2) to[in=230,out=250] (a14);
                \draw (a1) -- (a2);
                \draw[ultra thick,blue] (a2) -- (a3) -- (a4) -- (a5) -- (a6);
                \draw (a6) -- (a1) -- (a4);
                \draw[ultra thick,red] (a7) --node[left,draw=none,black]{$f$} (a10);
                \draw[ultra thick,red] (a8) -- (a9);
                \draw[ultra thick,red] (a9) -- (a10);
                \draw (a10) -- (a11) -- (a12) ;
                \draw (a7) -- (a8);
                \draw (a7) -- (a14);
                \draw[ultra thick,blue] (a12) --node[left,draw=none,black]{$f^{'}$} (a7) ;
                \draw[ultra thick,blue] (a12) -- (a13) -- (a14) ;
                
        \end{tikzpicture}
     
     \caption{}
    \label{fig:Illustation of Propositions- G_2a}
    \end{subfigure}%
     \
     \begin{subfigure}{0.4\textwidth}
        \centering
       \begin{tikzpicture}[every node/.style={draw=black, circle,scale=0.5}, scale=0.35]
               
                \node[label=below:{$h_0$}] (a1) at (5.5,0) {};
                \node[label=above:{$v_0$}] (a2) at (2.5,4) {};
                \node[label=above:{$ $}] (a3) at (4,4) {};
                \node[label=above:{$ $}] (a4) at (5.5,4) {};
                \node[label=above:{$ $}] (a5) at (7,4) {};
                \node[label=above:{$u_0$}] (a6) at (8.5,4) {};
                \node[label=above:{$h_1$}] (a7) at (13.5,4) {};
                \node[label=below:{$u_1=w^{'}$}] (a8) at (12,0) {};
                \node[label=below:{$ $}] (a9) at (13.5,0) {};
                \node[label=below:{$w^{''}$}] (a10) at (15,0) {};
                \node[label=below:{$ $}] (a11) at (16.5,0) {};
                \node[label=below:{$w$}] (a12) at (18,0) {};
                \node[label=below:{$ $}] (a13) at (19.5,0) {};
                \node[label=below:{$v_1$}] (a14) at (21,0) {};
                \draw (a8) -- (a1);
                \draw[ultra thick,blue] (a1) --node[above,draw=none,black]{$\alpha_0$} (a7);
                \draw (a7) --node[above,draw=none,black]{$\alpha_1$} (a6);
                \draw[ultra thick,blue] (a2) to[in=230,out=250] (a14);
                \draw[ultra thick,blue] (a1) -- (a2);
                \draw (a2) -- (a3) -- (a4) -- (a5) -- (a6) -- (a1) -- (a4);
                \draw[ultra thick,red] (a7) --node[left,draw=none,black]{$f^{'}$} (a8) -- (a9) -- (a10);
                \draw (a10) -- (a11) -- (a12);
                \draw[ultra thick,blue] (a12) -- (a13) -- (a14);
                \draw (a7) -- (a14);
                \draw[ultra thick,blue] (a12) --node[below,draw=none,black]{$f$} (a7);
                \draw[ultra thick,red] (a7) --node[left,draw=none,black]{$f^{''}$} (a10) ;
                
        \end{tikzpicture}
    \caption{}
    \label{fig:Illustation of Propositions- G_2b}
    \end{subfigure}%
    \caption{illustration for the proofs of Propositions~\ref{prop:g_2_osculating_bicycle_1} and \ref{prop:g_2_osculating_bicycle_2}}
    \label{fig:Illustation of Propositions- G_2}
\end{figure}
The next proposition deals with vertices of degree five or more.

\begin{prop}{\sc [Osculating Bicycles in Double Half Biwheels~-~II]}\newline
\label{prop:g_2_osculating_bicycle_2}
    Let $J \in \mathcal{G}_2$ and let $x$ denote a vertex of degree five or more. Let $\alpha_0$ and~$\alpha_1$ denote removable doubleton edges in $\partial(x)$ and let $f, f^{'},f^{''} \in \partial(x) - \{\alpha_0, \alpha_1\}$ so that $\alpha_1,\alpha_0, f^{'},f^{''},f$ appear in this cyclic order in the planar embedding of $J$. Then there exists an even osculating bicycle $(Q,Q^{'})$ such that $\alpha_0,f \in E(Q)$ and $f^{'},f^{''} \in E(Q^{'})$.
\end{prop}
\begin{proof}
We adopt notation from the definition of double \bw s; thus $x=h_1$. Let $f:=h_1w$, $f^{'}:=h_1w^{'}$ and $f^{''}:=h_1w^{''}$. We display an even osculating bicycle $(Q,Q^{'})$ whose each constituent cycle meets each of the two removable doubletons in zero or two edges as follows: $Q:=h_1fw(H_1-h_1)v_1\beta_0v_0h_0\alpha_0h_1$ and $Q^{'}:=h_1f^{'}w^{'}(H_1 - h_1)w^{''}f^{''}h_1$. See Figure~\ref{fig:Illustation of Propositions- G_2b} for an example.
\end{proof}

\subsection{Hexagon \bw s \texorpdfstring{$\mathcal{G}_{3}$}{}}
\label{sec:Hexagon-half-biwheel}
This section introduces our last family of graphs denoted as $\mathcal{G}_3$. A member of $\mathcal{G}_3$, say $J$, is obtained from the disjoint union of a hexagon (that is, a $6$-cycle) $H_0:=a_0a_1a_2a_3a_4a_5a_0$ and a \bw~$H_1$ by adding the following edges: $\alpha_0:=h_1a_4$, $\beta_0:=v_1a_1$, $\alpha_1:=a_3h_1$ and $\beta_1:=a_0u_1$. Figure~\ref{fig:a member of g_3} depicts two examples where $H_0$ and $H_1$ are shown in blue. Observe that if $H_1$ is isomorphic to $K_2$ then $J$ is the graph $R_8^{-}$.

\begin{figure}[!htb]
    \centering
     \begin{subfigure}{0.5\textwidth}
        \centering
      \begin{tikzpicture}[every node/.style={draw=black, circle,scale=0.7}, scale=0.5]
           
            \node[label=left:{$a_0$}] (a1) at (0,4) {};
            \node[label=below:{$u_1$}] (a2) at (1.5,4) {};
            \node[label=below:{$ $}] (a3) at (3,4) {};
            \node[label=below:{$ $}] (a4) at (4.5,4) {};
            \node[label=below:{$ $}] (a5) at (6,4) {};
            \node[label=below:{$v_1$}] (a8) at (7.5,4) {};
            \node[label=right:{$a_1$}] (a9) at (9,4) {};

            \node[label=left:{$a_4$}] (a10) at (0,8) {};
            \node[label=above:{$h_1$}] (a11) at (4.5,8) {};
            \node[label=right :{$a_3$}] (a12) at (9,8) {};

            \node[label=left:{$a_5$}] (a13) at (0,6){};
            \node[label=right:{$a_2$}] (a14) at (9,6){};
            \draw[blue] (a1) -- (a13) -- (a10);
            \draw[blue] (a9) -- (a14) -- (a12);
            \draw[thick] (a10) -- (a11);
            \draw[thick]  (a11) -- (a12);
            \draw[blue] (a1) to[in=240,out=300] (a9);
            \draw[blue] (a10) to[in=120,out=70] (a12);       
            \draw[thick] (a1) -- (a2); \draw[blue] (a2) -- (a3) -- (a4);
            \draw[blue] (a4) -- (a5) -- (a8);
            \draw[thick] (a8) -- (a9);
            \draw[blue] (a11) -- (a2);
            \draw[blue] (a11) -- (a4);
            \draw[blue] (a11) -- (a8);

    \end{tikzpicture}
     \caption{}
    
    \end{subfigure}%
     \begin{subfigure}{0.5\textwidth}
        \centering
      \begin{tikzpicture}[every node/.style={draw=black, circle,scale=0.7}, scale=0.5]
           
            \node[label=left:{$a_0$}] (a1) at (0,4) {};
            \node[] (a2) at (3,4) {};
            \node[draw = none,label=above:{$u=v$}] (a22) at (3,2.1) {};
            \node[label=right:{$a_1$}] (a9) at (6,4) {};

            \node[label=left:{$a_4$}] (a10) at (0,8) {};
            \node[label=above:{$h_1$}] (a11) at (3,8) {};
            \node[label=right :{$a_3$}] (a12) at (6,8) {};

            \node[label=left:{$a_5$}] (a13) at (0,6){};
            \node[label=right:{$a_2$}] (a14) at (6,6){};

            \draw[blue] (a1) -- (a13) -- (a10);
            \draw[blue] (a9) -- (a14) -- (a12);
            \draw[blue] (a2) -- (a11);

            \draw[thick] (a1) -- (a2) -- (a9);
            
            \draw[thick] (a10) -- (a11) -- (a12);
            \draw[blue] (a1) to[in=240,out=300] (a9);
            \draw[blue] (a10) to[in=120,out=70] (a12);

    \end{tikzpicture}
    \caption{}
        \label{fig:R_8-}
    \end{subfigure}%
    \caption{hexagon \bw s}
    \label{fig:a member of g_3}
\end{figure}

The reader may easily verify that each member of $\mathcal{G}_3$, say $J$, is also near-bipartite (\mc) graph and $R_0=\{\alpha_0,\beta_0\}$, $R_1=\{\alpha_1,\beta_1\}$ are the only removable doubletons. Furthermore, the components of $J-R_0-R_1$ are precisely the $6$-cycle~$H_0$ and the \bw~$H_1$. Using ideas similar to the ones in the proof of Proposition~\ref{prop:G_0-ce}, the reader may verify the following.

\begin{prop}
\label{prop:G_3-ce}
    Each member of $ \mathcal{G}_{3}$ is \ce.\qed
\end{prop}
    

We now proceed to make a couple of observations pertaining to the existence of even cycles containing a specified cubic vertex but not containing a specified neighbor; the first of these considers the case in which the cubic vertex belongs to the hexagon. We leave their proofs as exercises for the reader.

\begin{prop}{\sc [Even Cycles in Hexagon Half Biwheels - I]}\newline
\label{prop:G_3_cubic_vertex_six}
    Let $J \in \mathcal{G}_3$ and let $R_0,R_1$ denote the removable doubletons of~$J$.
    Let $H_0$ denote the component of $J-R_0-R_1$ that is isomorphic to a $6$-cycle. Let $x$ denote a cubic vertex such that $x \in V(H_0)$, and $w$ denote any neighbor of $x$. Then $J-w$ has an even cycle that contains $x$.\qed
\end{prop}

The next proposition considers the case in which the cubic vertex belongs to the \bw.

\begin{prop}{\sc [Even Cycles in Hexagon Half Biwheels - II]}\newline
\label{prop:G_3_cubic_vertex_hw}
    Let $J \in \mathcal{G}_3$ and let $R_0,R_1$ denote the removable doubletons of~$J$.
    Let $H_1$ denote the component of $J-R_0-R_1$ that is a \bw. Let $x$ denote a cubic vertex such that $x \in V(H_1)$ and $w$ denote a neighbor of $x$ in $J$ so that (i) either $w \notin V(H_1)$ or (ii)~$w \in V(H_1)$ and $d_{J}(w)=2$. Then $J-w$ has an even cycle that contains $x$.\qed
\end{prop}
\begin{figure}[!htb]
    \centering
     \begin{subfigure}{0.5\textwidth}
        \centering
          \begin{tikzpicture}[every node/.style={draw=black, circle,scale=0.7}, scale=0.5]
               
                \node[label=left:{$a_0$}] (a1) at (0,4) {};
                \node[label=below:{$u_1$}] (a2) at (1.5,4) {};
                \node[label=below:{$ $}] (a3) at (3,4) {};
                \node[label=below:{$ $}] (a4) at (4.5,4) {};
                \node[label=below:{$ $}] (a5) at (6,4) {};
                \node[label=below:{$v_1$}] (a8) at (7.5,4) {};
                \node[label=right:{$a_1$}] (a9) at (9,4) {};
    
                \node[label=left:{$a_4$}] (a10) at (0,8) {};
                \node[label=above:{$h_1$}] (a11) at (4.5,8) {};
                \node[label=right :{$a_3$}] (a12) at (9,8) {};
    
                \node[label=left:{$a_5$}] (a13) at (0,6){};
                \node[label=right:{$a_2$}] (a14) at (9,6){};
                \draw[ultra thick,blue] (a1) -- (a13) -- (a10);
                \draw[ultra thick,red] (a9) -- (a14) -- (a12);
                \draw[ultra thick,blue] (a10) --node[above,draw=none,black]{$\alpha_0$} (a11);
                \draw[ultra thick,red]  (a11) --node[above,draw=none,black]{$\alpha_1$} (a12);
                \draw (a1) to[in=240,out=300] (a9);
                \draw (a10) to[in=120,out=70] (a12);       
                \draw[ultra thick,blue] (a1) -- (a2) -- (a3) -- (a4);
                \draw (a4) -- (a5) -- (a8);
                \draw[ultra thick,red] (a8) -- (a9);
                \draw (a11) -- (a2);
                \draw[ultra thick,blue] (a11) --node[left,draw=none,black]{$f$} (a4);
                \draw[ultra thick,red] (a11) --node[above,draw=none,black]{$f^{'}$} (a8);

        \end{tikzpicture}
    \caption{}
        \label{fig:Illustation of Propositions- G_3a}
    \end{subfigure}%
     \begin{subfigure}{0.4\textwidth}
        \centering
         \begin{tikzpicture}[every node/.style={draw=black, circle,scale=0.7}, scale=0.5]
               
                \node[label=left:{$a_0$}] (a1) at (0,4) {};
                \node[label=below:{$u_1$}] (a2) at (1.5,4) {};
                \node[label=below:{$ $}] (a3) at (3,4) {};
                \node[label=below:{$v_1$}] (a8) at (4.5,4) {};
                \node[label=right:{$a_1$}] (a9) at (6,4) {};
    
                \node[label=left:{$a_4$}] (a10) at (0,8) {};
                \node[label=above:{$h_1$}] (a11) at (3,8) {};
                \node[label=right :{$a_3$}] (a12) at (6,8) {};
    
                \node[label=left:{$a_5$}] (a13) at (0,6){};
                \node[label=right:{$a_2$}] (a14) at (6,6){};
    
                \draw (a1) -- (a13) -- (a10);
                \draw (a9) -- (a14) -- (a12);
                \draw[ultra thick,blue] (a2) --node[above,draw=none,black]{$f$} (a11) --node[above,draw=none,black]{$f^{'}$} (a8);
    
                \draw[ultra thick,blue] (a1) -- (a2);
                \draw (a2) -- (a3) -- (a8);
                \draw[ultra thick,blue] (a8) -- (a9);
                \draw[ultra thick,red] (a10) --node[above,draw=none,black]{$\alpha_0$} (a11) --node[above,draw=none,black]{$\alpha_1$} (a12);
                \draw[ultra thick,blue] (a1) to[in=240,out=300] (a9);
                \draw[ultra thick,red] (a10) to[in=120,out=70] (a12);

        \end{tikzpicture}
    \caption{}
        \label{fig:Illustation of Propositions- G_3b}
    \end{subfigure}%
     \caption{illustration for the proof of Proposition~\ref{prop:g_3_osculating_bicycle}}
    \label{fig:Illustation of Propositions- G_3}
\end{figure}
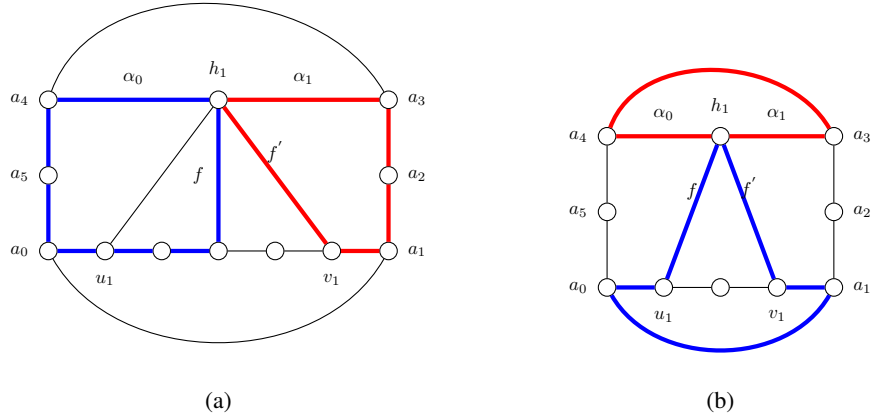

Finally, we discuss the existence of specific osculating bicycles in hexagon \bw s.

\begin{prop}{\sc [Osculating Bicycles in Hexagon Half Biwheels]}\newline
\label{prop:g_3_osculating_bicycle}
    Let $J \in \mathcal{G}_3$ and let $x$ denote a vertex of degree four or more. Let $\alpha_0$ and~$\alpha_1$ denote removable doubleton edges in $\partial(x)$ and let $f, f^{'} \in \partial(x) - \{\alpha_0,\alpha_1\}$ so that $\alpha_0,\alpha_1, f^{'},f$ appear in this cyclic order in the planar embedding of $J$. Then there exist:
    \begin{enumerate}[(i)]
        \item an odd osculating bicycle $(Q,Q^{'})$ such that $\alpha_0,\alpha_1 \in E(Q)$ and $f,f^{'} \in E(Q^{'})$, and
        \item an odd osculating bicycle $(C,C^{'})$ such that $\alpha_0,f \in E(C)$ and $\alpha_1,f^{'} \in E(C^{'})$.
    \end{enumerate}
\end{prop}
\begin{proof}
We adopt notation from the definition of hexagon \bw s; thus $x=h_1$. Let $f:=h_1w$ and $f^{'}:=h_1w^{'}$.

To prove (i), we display an odd osculating bicycle $(Q,Q^{'})$, whose each constituent cycle uses precisely one edge from each of the two removable doubletons of $J$ as follows: $Q:=h_1w(H_1-h_1)u_1\beta_1a_0a_5a_4\alpha_0h_1$ and $Q^{'} := h_1w^{'}(H_1-h_1)v_1\beta_0a_1a_2a_3\alpha_1h_1$. Figure~\ref{fig:Illustation of Propositions- G_3a} shows an illustration.
 
Likewise, to prove (ii), we display an odd osculating bicycle $(C,C^{'})$, whose each constituent cycle uses precisely one edge from each of the two removable doubletons of $J$ as follows: $C:=h_1\alpha_1a_3a_4\alpha_0h_1$ and $C^{'} := h_1w(H_1-h_1)u_1\beta_1a_0a_1\beta_0v_1(H_1-h_1)w^{'}h_1$. Figure~\ref{fig:Illustation of Propositions- G_3b} shows an illustration.
\end{proof}
\subsection{Main Theorem: planar \ce~ irreducible graphs}
\label{section-6}
The following observation pertaining to all of the families defined in the previous section may be easily verified by the reader; see \cite[Theorem 10.7, Corollary 10.8]{bomu08}.

\begin{prop}
\label{prop:planar-3-connected}
    Let $J \in \mathcal{G}_{0}\cup\mathcal{G}_{1}\cup\mathcal{G}_{2}\cup\mathcal{G}_{3}$ and let $x$ denote a vertex of degree four or more. Then $J - x$ is a $2$-connected (planar) graph. Consequently, the neighbors of $x$ lie on a cycle of $J-x$. \qed
\end{prop}

We are now ready to state and prove the Main Theorem.

\begin{thm}{\sc[Planar Cycle-Extendable Irreducible Graphs]}\newline
\label{thm:characterize-nonbipartite-ce-graphs}
A planar irreducible \mcg~$G$ is \ce\ if and only if $G$ belongs to $\{K_2\} \cup  \mathcal{G}_{0} \cup \mathcal{G}_{1} \cup \mathcal{G}_{2} \cup \mathcal{G}_{3}$.
\end{thm}
\begin{proof}
 Propositions~\ref{prop:G_0-ce}, \ref{prop:G_1-ce},  \ref{prop:G_2-ce} and \ref{prop:G_3-ce} prove the reverse implication. For the forward implication, let $G$ be a planar irreducible \mcg~that is \ce. We proceed by induction on the number of edges. Firstly, if $\delta(G)\geq 3$ then, by Corollary~\ref{DegreeThree-Graph} and Corollary~\ref{cor:Planar Cycle-Extendable Bricks}, $G$ is either a wheel or a prism; thus $G \in \mathcal{G}_0 \cup \mathcal{G}_1$.

Now suppose that $G$ has a vertex of degree two, say $x_0$, and let $x_1$ and $x_2$ denote its neighbors. Let $J^{'} := G/x_0$ and let $x$ denote its bicontraction vertex. Since $G$ is irreducible, $d_{J^{'}}(x) \geq 4$. Since $G$ is \ce, by the first part of Corollary~\ref{cor:J is 2-stable and parallel edge}, $J^{'}$ has at most two multiple edges. Clearly, multiple edges (if any) are incident at $x$. In case $J^{'}$ is not simple we use $e$ and $e^{'}$ to denote its multiple edges. We let $J$ denote the underlying simple graph of $J^{'}$. Furthermore, either $J = J^{'}$ or otherwise we adjust notation so that $J = J^{'} - e^{'}$. Thus, $d_{J}(x) \geq 3$. Since $G$ is irreducible, by the second part of Corollary~\ref{cor:J is 2-stable and parallel edge}, $J$ is also irreducible. 

Clearly, $J$ and $J^{'}$ are planar \mcg s. By Lemma~\ref{prp:cycle-extendability-inherited-through-tight-cuts}, they are also \ce. Since $J$ is irreducible and $d_J(x) \geq 3$, by the induction hypothesis, ${J}$ belongs to $\mathcal{G}_{0}\cup\mathcal{G}_{1}\cup\mathcal{G}_{2}\cup\mathcal{G}_{3}$.

We will divide the proof into cases depending on the degree of vertex $x$ and whether $J$ belongs to $\mathcal{G}_{0}$, $\mathcal{G}_{1}$, $ \mathcal{G}_{2}$ or $ \mathcal{G}_{3}$. Within each case, we will consider various subcases. In each subcase, we will either display an osculating bicycle $(Q, Q^{'})$ in $J^{'}$ such that neither of its constituent cycles is a cycle in $G$ and invoke Lemma \ref{lem:osculating-bicycle} to arrive at a contradiction (by inferring that $G$ is not \ce), or otherwise conclude that $G$ belongs to $\mathcal{G}_{0}, \mathcal{G}_{1}, \mathcal{G}_{2}$ or $ \mathcal{G}_{3}$. Depending on the family that $J$ belongs to, we shall adopt the following notation.

\begin{Not}
\label{not:1}
    If $J\in \mathcal{G}_0\cup\mathcal{G}_2\cup\mathcal{G}_3$, then we let $R_0,R_1,\ldots,R_k$ denote all of the removable doubletons of $J$, and $H_0,H_1,\ldots,H_k$ denote the components of $G-R_0-R_1-\ldots-R_k$. Furthermore, if $J \in \mathcal{G}_3$, then we let $H_0$ denote the component that is a $6$-cycle.
\end{Not}

\begin{Not}
\label{not:2}
    If $J \in \mathcal{G}_1-K_4$, then we let $E_3$ denote the set of edges whose both ends are cubic, $h$ denote the unique cut vertex of $J-E_3$, and $H_0,H_1,\ldots,H_k$ denote the (\bw) blocks of $J-E_3$ so that they appear in this cyclic order (around the hub $h$) in the planar embedding of $J$ and adjust notation so that $E_3=\{v_iu_{i+1}: v_i \in V(H_i)~{\rm and }~u_{i+1} \in V(H_{i+1})\}$.
\end{Not}
{\bf Case 1: $d_{J}(x) = 3$}.

Observe that $J^{'}$ is obtained from $J$ by adding a multiple edge $e^{'}$ incident with $x$.
Let $e=e^{'}=xw$ in $J^{'}$ and let $Q^{'}$ denote $2$-cycle $ee^{'}$. Observe that, since $d_{J^{'}}(x)=4$ and since $G$ is simple, there is a unique way to bisplit $x$ and obtain the graph $G$ from $J^{'}$. Consequently, if there exists an even cycle $Q$ in $J^{'}$ such that $(Q, Q^{'})$ is an osculating bicycle in $J^{'}$, then neither $Q$~nor~$Q^{'}$ is a cycle in $G$; in this case, by Lemma \ref{osculating-bicycle}, we arrive at a contradiction to the assumption that $G$ is \ce. Observe that $Q$ exists if and only if $J-w$ has an even cycle that contains $x$. From the preceding discussion, it suffices to consider only those cases in which $J-w$ has no even cycle that contains $x$. In the following two paragraphs, we shall heavily exploit this observation.

If $J \in \mathcal{G}_l$ where $l \in \{0,2,3\}$, we adopt Notation~\ref{not:1}. If $l=3$, we invoke Proposition~\ref{prop:G_3_cubic_vertex_six} to infer that $x \notin V(H_0)$. Now, we may adjust notation so that $x \in V(H_1)$.  Depending on the value of $l$, we invoke Proposition \ref{prop:G_0_cubic_vertex}, or \ref{prop:G_2_cubic_vertex}, or \ref{prop:G_3_cubic_vertex_hw}, to infer that $d_J(w) \geq 3$ and that $w$ is a hub of the same \bw~$H_1$. Since the bisplitting is unique (as noted earlier), the reader may easily verify that $G$ belongs to the same family $\mathcal{G}_l$.

Now consider the case in which $J \in \mathcal{G}_1$. If $J = K_4$ then the reader may easily verify that $G=W_5^{-}$. Now suppose that $J \neq K_4$ and adopt Notation~\ref{not:2}. We invoke Proposition~\ref{prop:G_1_cubic_vertex} to infer that either $w=h$, or otherwise, $|E_3|=3$ and $J-E_3-h$ has precisely two isolated vertices --- namely, $x$~and~$w$. If $w=h$, then using the fact that bisplitting of $x$ is unique, one may easily see that $G \in \mathcal{G}_1$. We consider the remaining case below.
\begin{figure}[!htb]
    \centering
      \begin{subfigure}{0.4\textwidth}
        \centering
         \begin{tikzpicture}[every node/.style={draw=black, circle,scale=0.7}, scale=0.5]
                    \node (a6) at (0,1.5) {};
                    \node (a7) at (-0.5,5.75) {};
                    \node (a12) at (-0.3,3.75) {};
                    \node (a13) at (1,6.75) {};
                    \node (a8) at (3,8) {};
                    \node[label=right:{$x$}] (a9) at (6.5,5.75) {};
                    \node[label=right:{$w$}] (a10) at (5.75,1.5) {};
                    \node[label=below:{$h$}] (a1) at (3,4.5){};
                    \draw (a6) -- (a12);
                    \draw (a12) -- (a7);
                    \draw (a7) -- (a1);
                    \draw (a1) -- (a6);
                    \draw (a1) -- (a8);
                    \draw (a1) -- (a9);
                    \draw (a1) -- (a10);
                    \draw (a7) -- (a13) -- (a8);
                    \draw (a8) -- (a9);
                    \draw[] (a9) -- (a10);
                    \draw[] (a9) to[in=40,out=300]  (a10);
                    \draw (a10) -- (a6);
        \end{tikzpicture}
    \caption{J}
        
    \end{subfigure}
    \hspace*{50pt}
      \begin{subfigure}{0.4\textwidth}
        \centering
        \begin{tikzpicture}[every node/.style={draw=black, circle,scale=0.7}, scale=0.5]
                    \node (a6) at (0,1.5) {};
                    \node (a7) at (-0.5,5.75) {};
                    \node (a12) at (-0.3,3.75) {};
                    \node (a13) at (1,6.75) {};
                    \node (a8) at (3,8) {};
                    \node[label=right:{$x_1$}] (a19) at (6,4.75) {};
                    \node[label=left:{$x_0$}] (a9) at (6.5,5.75) {};
                    \node[label=right:{$x_2$}] (a20) at (6,6.75) {};
                    \node[label=right:{$w$}] (a10) at (5.75,1.5) {};
                    \node[label=below:{$h$}] (a1) at (3,4.5){};
                    \draw[matching] (a6) -- (a12);
                    \draw[matching] (a12) -- (a7);
                    \draw[matching] (a7) -- (a1);
                    \draw[matching] (a1) -- (a6);
                    \draw[matching] (a1) -- (a8);
                    \draw (a1) -- (a19);
                    \draw (a1) -- (a10);
                    \draw[matching] (a7) -- (a13) -- (a8);
                    \draw (a8) -- (a20);
                    \draw[matching] (a19)-- (a9) -- (a20);
                    \draw[matching] (a10) -- (a19);
                    \draw[matching](a20) to[in=40,out=330]  (a10);
                    \draw (a10) -- (a6);
        \end{tikzpicture}
    
    \caption{G}
    
    \end{subfigure}
    \caption{the case in which $J \in \mathcal{G}_1$ and $G \in \mathcal{G}_2$}
    \label{fig:graph_g_2_main_proof}
\end{figure}
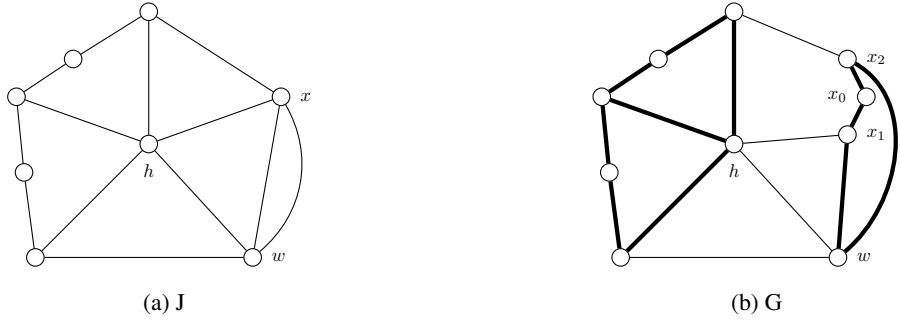\\
Suppose that $|E_3|=3$ and that $J-E_3-h$ has precisely two isolated vertices --- namely, $x$~and~$w$. Based on the description of $\mathcal{G}_1$, this is equivalent to saying that $k=2$ and precisely two of the \bw s $H_0$, $H_1$ and $H_2$ are isomorphic to $K_2$. We adjust notation so that each of $H_1$ and $H_2$ is isomorphic to $K_2$ and that $x \in V(H_1)$ and $w \in V(H_2)$. Since the bisplitting is unique, observe that $H:=G$[$w,x_1,x_0,x_2$] is isomorphic to~$C_4$, and that $G$ is a member of $\mathcal{G}_2$ ---  wherein $H$ (with hub $w$) and $H_0$ (with hub $h$) are the \bw s as per the definition of double \bw s in Section~\ref{sec:Double-half-biwheel}. See Figure \ref{fig:graph_g_2_main_proof}.

{\bf Case 2: $d_{J}(x) \geq 4$}.

Note that, unlike the previous case, we must consider all possible bisplittings of the vertex~$x$ that lead to a planar graph (instead of just one bisplitting). Observe that $\partial_{J^{'}}(x) = \partial_G(\{x_0,x_1,x_2\})$. It follows from Proposition~\ref{prop:planar-3-connected} that the cyclic order of these edges must be the same in $G$ around the set $\{x_0,x_1,x_2\}$ as the cyclic order in $J^{'}$; otherwise, it is easy to see that the resulting graph has a subdivision of $K_{3,3}$.

 We consider subcases depending on whether $J$ belongs to $\mathcal{G}_{0}$ $\cup$ $\mathcal{G}_{2}$ $\cup$ $ \mathcal{G}_{3}$ or to $ \mathcal{G}_{1}$. Note that an osculating bicycle in $J$ also exists in $J^{'}$. It is for this reason that, when applicable, we simply display an osculating bicycle in $J$ with the desired properties in order to arrive at a contradiction as discussed earlier.

\noindent{\bf Case 2.1: $J \in \mathcal{G}_0 \cup \mathcal{G}_2 \cup \mathcal{G}_3$.}

We adopt Notation~\ref{not:1}. Since $d_{J}(x) \geq 4$, we may adjust notation so that $x$ is a hub of $H_1$ that is not isomorphic to $K_2$. Let $\alpha_0$ and $\alpha_1$ denote the removable doubleton edges in~$\partial_{J}(x)$. Since $G$ is obtained from $J^{'}$ by bisplitting the vertex $x$, the edges $\alpha_0$ and $\alpha_1$ may or may not be adjacent in $G$; we consider these cases separately. 

First suppose that $\alpha_0$ and $\alpha_1$ are adjacent in $G$ and adjust notation so that $\alpha_0,\alpha_1  \in \partial_{G}(x_1)$. We choose $f,f^{'} \in \partial_{G}(x_2)-x_2x_0$ so that $\alpha_0,\alpha_1, f^{'},f$ appear in this cyclic order in the planar embedding of $J$. Depending on whether $J$ belongs to $\mathcal{G}_0,\mathcal{G}_2$ or $\mathcal{G}_3$, we invoke Proposition \ref{prop:g_0_osculating_bicycle} (ii), \ref{prop:g_2_osculating_bicycle_1} or \ref{prop:g_3_osculating_bicycle} (ii), in order to locate the desired osculating bicycle $(Q,Q^{'})$ in $J$. 

Now suppose that $\alpha_0$ and $\alpha_1$ are nonadjacent in $G$ and adjust notation so that $\alpha_0\in \partial(x_1)$ and $\alpha_1\in \partial(x_2)$. We consider two subcases depending on whether $J \in \mathcal{G}_0 \cup \mathcal{G}_3$ or $J \in \mathcal{G}_2$. If $J \in \mathcal{G}_0 \cup \mathcal{G}_3$, we choose $f \in \partial_{G}(x_1) - x_1x_0 - \alpha_0$ and $f^{'} \in \partial_{G}(x_2)-x_2x_0 - \alpha_1$ so that $\alpha_0, \alpha_1, f^{'},f$ appear in this cyclic order in the planar embedding of $J$. Depending on whether $J$ belongs to $\mathcal{G}_0$ or $\mathcal{G}_3$, we invoke Proposition \ref{prop:g_0_osculating_bicycle} (i) or \ref{prop:g_3_osculating_bicycle} (i), in order to locate the desired osculating bicycle $(Q,Q^{'})$ in $J$. Henceforth $J \in \mathcal{G}_2$; we consider two subcases depending on whether $d_{J}(x) = 4$ or $d_{J}(x) \geq 5$.

  Let us first suppose that $d_{J}(x) = 4$. Equivalently, $H_1$ is isomorphic to a $4$-cycle, say $h_1u_1yv_1h_1$, where $h_1:=x$. Since $\alpha_0 \in \partial_G(x_1)$ and $\alpha_1 \in \partial_G(x_2)$, observe that $x_1u_1yv_1x_2x_0x_1$ is a $6$-cycle in $G$ and that $G\in \mathcal{G}_3$ as per the definition of hexagon \bw s in Section~\ref{sec:Hexagon-half-biwheel}; see Figure \ref{fig:Illustration-bisplitting-g_2-g_3}.

\begin{figure}[!htb]
    \centering
    
      \begin{subfigure}{0.4\textwidth}
        \centering
        \begin{tikzpicture}[every node/.style={draw=black, circle,scale=0.66}, scale=0.45]
               
                \node[label=below:{$h_0$}] (a1) at (5.5,0) {};
                \node[label=above:{$v_0$}] (a2) at (2.5,4) {};
                \node[label=above:{$ $}] (a3) at (4,4) {};
                \node[label=above:{$ $}] (a4) at (5.5,4) {};
                \node[label=above:{$ $}] (a5) at (7,4) {};
                \node[label=above:{$u_0$}] (a6) at (8.5,4) {};
                \node[label=above:{$h_1=x$}] (a7) at (13.5,4) {};
                \node[label=below:{$u_1$}] (a8) at (12,0) {};
                \node[label=below:{$y$}] (a9) at (14,0) {};
                \node[label=below:{$v_1$}] (a10) at (16,0) {};
                \draw (a8) -- (a1);
                \draw (a1) -- (a7);
                \draw (a7) -- (a6);
                \draw (a2) to[in=240,out=240] (a10);
                \draw (a1) -- (a2);
                \draw (a2) -- (a3) -- (a4) -- (a5) -- (a6);
                \draw (a6) -- (a1) -- (a4);
                \draw (a7) -- (a8);
                \draw (a8) -- (a9);
                \draw (a9) -- (a10);
                \draw (a10) -- (a7) ;
                
        \end{tikzpicture}
    
    \caption{J}
        
    \end{subfigure}
    \hspace*{50pt}
      \begin{subfigure}{0.4\textwidth}
        \centering
       \begin{tikzpicture}[every node/.style={draw=black, circle,scale=0.7}, scale=0.5]
               
                \node[label=left:{$v_1$}] (a1) at (0,8) {};
                \node[label=above:{$v_0$}] (a2) at (1.5,8) {};
                \node[label=above:{$ $}] (a3) at (3,8) {};
                \node[label=above:{$ $}] (a4) at (4.5,8) {};
                \node[label=above:{$ $}] (a5) at (6,8) {};
                \node[label=above:{$u_0$}] (a8) at (7.5,8) {};
                \node[label=right:{$x_2$}] (a9) at (9,8) {};
    
                \node[label=left:{$u_1$}] (a10) at (0,4) {};
                \node[label=below:{$h_0$}] (a11) at (4.5,4) {};
                \node[label=right :{$x_1$}] (a12) at (9,4) {};
    
                \node[label=left:{$y$}] (a13) at (0,6){};
                \node[label=right:{$x_0$}] (a14) at (9,6){};
                \draw (a1) -- (a13) -- (a10);
                \draw (a9) -- (a14) -- (a12);
                \draw (a10) -- (a11);
                \draw  (a11) -- (a12);
                \draw (a1) to[in=120,out=60] (a9);
                \draw (a10) to[in=240,out=290] (a12);       
                \draw (a1) -- (a2) -- (a3) -- (a4);
                \draw (a4) -- (a5) -- (a8);
                \draw (a8) -- (a9);
                \draw (a11) -- (a2);
                \draw (a11) -- (a4);
                \draw (a11) -- (a8);

        \end{tikzpicture}
    
    \caption{G}
        
    \end{subfigure}
    \caption{the case in which $J \in \mathcal{G}_2$ and $G \in \mathcal{G}_3$}
    \label{fig:Illustration-bisplitting-g_2-g_3}
\end{figure}
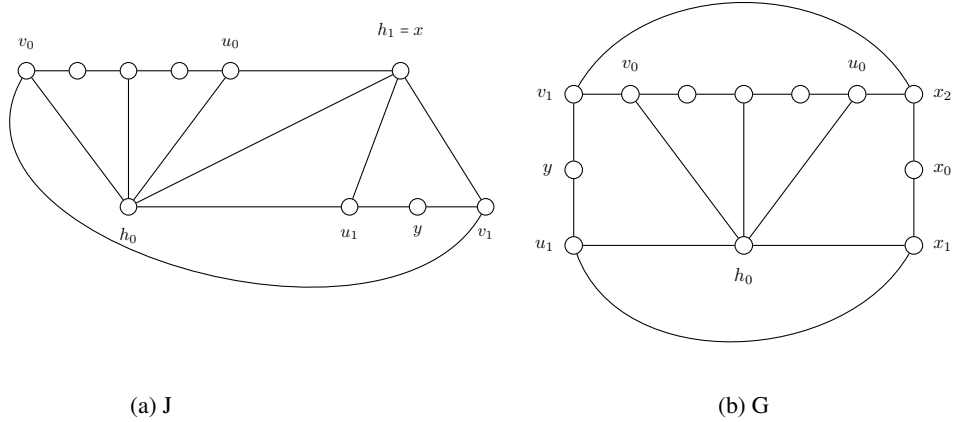
Lastly, suppose that $d_{J}(x) \geq 5$; consequently, at least one of $x_1$ and $x_2$ has degree four or more. Adjust notation so that $d_G(x_2) \geq 4$. Now, we may choose $f^{'} \in \partial_G(x_1)-x_1x_0 - \alpha_0$ and distinct $f^{''},f \in \partial_G(x_2)-x_2x_0-\alpha_1$ so that $\alpha_1, \alpha_0, f^{'},f^{''},f$ appear in this cyclic order in the planar embedding of $J$. We invoke Proposition \ref{prop:g_2_osculating_bicycle_2} to locate the desired osculating bicycle $(Q,Q^{'})$ in $J$.\\

\noindent{\bf Case 2.2: $J \in \mathcal{G}_1.$}

We adopt Notation~\ref{not:2}. Since $d_J(x)\geq 4$, the vertex $x$ is precisely the hub $h$ of $J$. Note that, for each $0 \leq i \leq k$, the edge set of the \bw~$H_i$ (in $J$) may or may not form a \bw~in $G$. Furthermore, $E(H_i)$ forms a \bw~in $G$ if and only if all edges in $E(H_i) \cap \partial_J(x)$ are incident with (precisely) one of $x_1$ and $x_2$ (in G); in this case, we say that the \bw~$H_i$ {\em remains intact}; otherwise we say that the \bw~$H_i$ {\em is destroyed}. It follows from planarity that the number of \bw s that are destroyed is either zero, two or one; we shall consider these cases separately in that order.

First we consider the case in which zero \bw s are destroyed. We may adjust notation so that there exists $0\leq r < k$ such that: (i) for each $0\leq i \leq r$, all edges in $E(H_i) \cap \partial_J(x)$ are incident with $x_1$ (in $G$), and likewise (ii) for each $r < i \leq k$, all edges in $E(H_i) \cap \partial_J(x)$ are incident with $x_2$. It follows from the irreducibility of $G$ that the four edges $xu_0,xv_r,xu_{r+1}$ and $xv_{k}$ are pairwise distinct. We invoke Proposition~\ref{prop:G_1_non_cubic_vertex} (i) to locate the desired osculating bicycle $(Q,Q^{'})$ in $J$.

Next we consider the case in which two \bw s, say $H_i$ and $H_r$, are destroyed. Let $e$ denote an edge in $E(H_i) \cap \partial_J(x)$ that is incident with $x_1$ (in $G$) and $f$ denote an edge in $E(H_i) \cap \partial_J(x)$ that is incident with $x_2$. Likewise, let $e^{'}$ denote an edge in $E(H_r) \cap \partial_J(x)$ that is incident with $x_1$ and $f^{'}$ denote an edge in $E(H_r) \cap \partial_J(x)$ that is incident with $x_2$. We invoke Proposition~\ref{prop:G_1_non_cubic_vertex} (ii) to locate the desired osculating bicycle $(Q,Q^{'})$ in $J$.

Lastly we consider the case in which precisely one \bw, say $H_0$, is destroyed. we may adjust notation so that there exists $0 < r \leq k$ such that: (i) for each $0 < i \leq r$, all edges in $E(H_i) \cap \partial_J(x)$ are incident with $x_1$ (in $G$), and likewise (ii) for each $r < i \leq k$, all edges in $E(H_i) \cap \partial_J(x)$ are incident with $x_2$. Unless $H_0$ is isomorphic to $C_4$, $k=2$, each of $H_1$~and~$H_2$ is isomorphic to $K_2$, and $r=1$, the reader may verify that it is possible to choose $e_0,e_1,e_r,e_{r+1} \in \partial_J(h)$ that satisfy the conditions stated in Proposition~\ref{prop:G_1_non_cubic_vertex} (iii) so that $e_0$~and~$e_1$ are nonadjacent in $G$ and, likewise, $e_r$ and $e_{r+1}$ are nonadjacent in $G$; we invoke Proposition~\ref{prop:G_1_non_cubic_vertex} (iii) to locate the desired osculating bicycle $(Q,Q^{'})$ in $J$.

Now suppose that $H_0$ is isomorphic to $C_4$, $k=2$, each of $H_1$~and~$H_2$ is isomorphic to $K_2$, and $r=1$. Observe that $J$ is the graph $W_5^{-}$ as per the labels shown in Figure~\ref{fig:w_5-}. First suppose that $J \neq J^{'}$; up to symmetry, either $e^{'} = hu_1$ or otherwise $e^{'} = hu_0$. In the former case, since $H_0$ is destroyed and $G$ is simple, $(Q:=H_0,Q^{'}:=ee^{'})$ is the desired even osculating bicycle in $J^{'}$. In the latter case, it follows from irreducibility and planarity that $(Q:=hv_0u_2u_1h,Q^{'}:=ee^{'})$ is the desired even osculating bicycle in $J^{'}$. Finally, if $J=J^{'}$, then it follows from the facts that $H_0$ is destroyed and $G$ is irreducible, that $G$ is the graph $R_8^{-} \in \mathcal{G}_3$ shown in Figure~\ref{fig:R_8-}. 

This completes our proof of the Main Theorem (\ref{thm:characterize-nonbipartite-ce-graphs}).
\end{proof}

We conclude our paper with the following corollaries of our Main Theorem (\ref{thm:characterize-nonbipartite-ce-graphs}) and Proposition~\ref{prop:irreducible-reduction}.
\begin{cor}
    A planar bipartite \mcg~$G$ is \ce~if and only if any irreducible graph H, obtained from $G$ by repeated applications of series and parallel reductions, is isomorphic to $K_2$.\qed
\end{cor}
\begin{cor}
    A planar nonbipartite \mcg~$G$ is \ce~if and only if any irreducible graph H, obtained from $G$ by repeated applications of series and parallel reductions, belongs to $\mathcal{G}_{0}\cup\mathcal{G}_{1}\cup\mathcal{G}_{2}\cup\mathcal{G}_{3}$.\qed
\end{cor}




\bibliographystyle{abbrvnat}
\bibliography{clm}

\end{document}